\documentclass[12pt,reqno]{amsart}
\date{May 6, 2010}      
\usepackage{fullpage}
\usepackage{latexsym,amsmath,amsfonts,amscd,amssymb}
\usepackage{lscape}
\usepackage{array}
\usepackage[all]{xy}
\DeclareFontFamily{OT1}{rsfs}{}
\DeclareFontShape{OT1}{rsfs}{n}{it}{<->rsfs10}{}
\DeclareMathAlphabet{\curly}{OT1}{rsfs}{n}{it}
\theoremstyle{plain}  
\newtheorem{theorem}{Theorem}[section]
\newtheorem*{theorem*}{Theorem}

\newtheorem{corollary}[theorem]{Corollary}
\newtheorem{lemma}[theorem]{Lemma}
\newtheorem{proposition}[theorem]{Proposition}

\newtheorem{definition}[theorem]{Definition}

\theoremstyle{definition}
\newtheorem*{convention*}{Convention}

\theoremstyle{remark}

\newtheorem{remark}[theorem]{Remark}

\newtheorem*{claim*}{Claim}

\numberwithin{equation}{section}
\newcommand{\suchthat}{\;|\;}
\newcommand{\abs}[1]{\lvert#1\rvert}
\renewcommand{\leq}{\leqslant}
\renewcommand{\le}{\leqslant}
\renewcommand{\geq}{\geqslant}
\renewcommand{\ge}{\geqslant}
\renewcommand{\setminus}{\smallsetminus}
\newcommand{\x}{\times}
\newcommand{\into}{\hookrightarrow}

\newcommand{\R}{\mathbb{R}}

\newcommand{\Z}{\mathbb{Z}}
\newcommand{\C}{\mathbb{C}}
\newcommand{\D}{\mathbb{D}}

\newcommand{\PP}{\mathbb{P}}

\newcommand{\cO}{\mathcal{O}}

\newcommand{\lie}{\mathfrak}
\newcommand{\ra}{\rightarrow}
\newcommand{\lra}{\longrightarrow}

\newcommand{\PSL}{\mathrm{PSL}}

\newcommand{\SU}{\mathrm{SU}}
\newcommand{\U}{\mathrm{U}}
\newcommand{\OO}{\mathrm{O}}
\newcommand{\CO}{\mathrm{CO}}
\newcommand{\GL}{\mathrm{GL}}
\newcommand{\SL}{\mathrm{SL}}

\newcommand{\SO}{\mathrm{SO}}
\newcommand{\Sp}{\mathrm{Sp}}

\newcommand{\ua}{\underline{a}}
\newcommand{\ub}{\underline{b}}
\DeclareMathOperator{\Jac}{Jac} 
\DeclareMathOperator{\Ad}{Ad} 
 \DeclareMathOperator{\tr}{tr}
\DeclareMathOperator{\rk}{rk} 
 
\DeclareMathOperator{\Hom}{Hom}

\DeclareMathOperator{\Lie}{Lie}
 \DeclareMathOperator{\Mat}{Mat}
\DeclareMathOperator{\Sym}{Sym} 

\newcommand{\noi}{\noindent}

\renewcommand{\phi}{\varphi}

\newcommand{\liem}{\mathfrak{m}}

\newcommand{\liemc}{\mathfrak{m}^{\mathbb{C}}}
\newcommand{\lieh}{\mathfrak{h}}
\newcommand{\liehc}{\mathfrak{h}^{\mathbb{C}}}
\newcommand{\lieg}{\mathfrak{g}}
\newcommand{\liep}{\mathfrak{p}}
\newcommand{\plie}{\mathfrak{p}}
\newcommand{\liegc}{\mathfrak{g}^{\mathbb{C}}}
\newcommand{\liesl}{\mathfrak{sl}}

\newcommand{\Inn}{\operatorname{Inn}}
\newcommand{\Out}{\operatorname{Out}}
\newcommand{\Aut}{\operatorname{Aut}}

\begin{document}

\title{Deformations of maximal representations in $\Sp(4,\R)$}

\author[S. B. Bradlow]{Steven B. Bradlow}
\address{Department of Mathematics \\
University of Illinois \\
Urbana \\
IL 61801 \\
USA } \email{bradlow@math.uiuc.edu}

\author[O. Garc{\'\i}a-Prada]{Oscar Garc{\'\i}a-Prada}
\address{Instituto de Ciencias Matem\'aticas 
CSIC-UAM-UC3M-UCM \\ Serrano 121 \\ 28006 Madrid \\ Spain}
\email{oscar.garcia-prada@uam.es}

\author[P. B. Gothen]{Peter B. Gothen}
\address{Departamento de Matem\'atica Pura \\
Faculdade de Ci\^encias, Universidade do Porto \\
Rua do Campo Alegre 687 \\ 4169-007 Porto \\ Portugal }
\email{pbgothen@fc.up.pt}

\thanks{
Members of VBAC (Vector Bundles on Algebraic Curves).
Second and Third authors partially supported by FCT / CSIC and
CRUP / Ministerio de
Educaci{\'o}n y Ciencia (Spain) through Portugal/Spain bilateral
programmes.
First and Second authors partially supported by Ministerio de
Educaci{\'o}n y Ciencia (Spain) through Project
MTM2004-07090-C03-01.
Third author partially supported by FCT (Portugal) through the Centro
de Matem\'atica da Universidade do Porto and the projects
POCTI/MAT/58549/2004, PTDC/MAT/099275/2008 and PTDC/MAT/098770/2008.
}


\begin{abstract}
We use Higgs bundles to answer the following question: When can a maximal
$\Sp(4,\R)$-representation of a surface group be deformed to a
representation which factors through a proper reductive subgroup of
$\Sp(4,\R)$?
\end{abstract}

\maketitle
\section{Introduction}

A good way to understand an object of study, as Richard Feynman
famously remarked\footnotemark\footnotetext{In his lecture 
``There's
plenty of room at the bottom'' (see \cite{feynman})}, is to ``just
look at the thing''.  In this paper we apply Feynman's method to answer
the following question: given a surface group representation in
$\Sp(4,\R)$, under what conditions can it be deformed to a
representation which factors through a proper reductive subgroup of
$\Sp(4,\R)$?

A surface group representation in a group $G$ is a homomorphism from
the fundamental group of the surface into $G$. For a surface of genus
$g\ge 2$, the moduli space of reductive surface group representations
into $G=\Sp(4,\R)$, denoted by $\mathcal{R}(\Sp(4,\R))$, has $3\cdot
2^{2g+1}+8g-13$ connected components (see
\cite{garcia-prada-mundet:2004,gothen:2001}).  The components are
partially labeled by an integer, known as the Toledo invariant, which
ranges between $2-2g$ and $2g-2$. If $\mathcal{R}_d$ denotes the
component with Toledo invariant $d$, then there is a homeomorphism
$\mathcal{R}_d \simeq \mathcal{R}_{-d}$ and except for the extremal
cases (i.e.\ $|d|=2g-2$) each $\mathcal{R}_d$ is connected.  In
contrast, the subspace of \textbf{maximal representations}
$\mathcal{R}^{max}=\mathcal{R}_{2g-2}$ have $3\cdot 2^{2g}+2g-4$
components.  These are our objects of study. The precise question we
answer is thus: {\it which maximal components contain representations
that factor through reductive subgroups of $\Sp(4,\R)$?}

One motivation for this question stems from the fundamental work of
Goldman \cite{goldman:1980,goldman:1988} and Hitchin
\cite{hitchin:1992}. Goldman showed that, in the case of $\PSL(2,\R)$,
the space of maximal representations coincides with Teichm\"uller
space, i.e., the space of Fuchsian representations.  
Using Higgs bundles, Hitchin
constructed distinguished components in the moduli space of reductive
representations in the split real form of any complex reductive
group. These components, known as \textbf{Hitchin components}, 
have been the subject of much interest, see for example
Burger--Iozzi--Labourie--Wienhard,
\cite{burger-iozzi-labourie-wienhard:2005}, Fock--Goncharov
\cite{fock-goncharov:2006}, Guichard-Wienhard
\cite{guichard-wienhard:2008} and La\-bou\-rie
\cite{labourie:2005,labourie:2006}.  

Moreover, the representations in these components factor through
homomorphisms from $\SL(2,\R)$ into the split real form.  In the case
of $\Sp(4,\R)$ there are $2^{2g}$ Hitchin components, all of which are
maximal and contain representations which factor through the
irreducible representation of $\SL(2,\R)$ in $\Sp(4,\R)$.  One is thus
led to ask whether the other $ 2^{2g+1}+2g-4$ components have similar
factorization properties.

In the case of $\Sp(4,\R)$ there are $2^{2g}$ Hitchin components.
They are projectively equivalent, in the sense that they project to a
unique Hitchin component in the moduli space for the projective
symplectic group $\mathrm{PSp}(4,\R)$.  The $\Sp(4,\R)$ Hitchin
components are all maximal and all contain representations which
factor through the {\it irreducible} representation of $\SL(2,\R)$ in
$\Sp(4,\R)$.  One is thus led to ask whether the other $2^{2g+1}+2g-4$
maximal components have similar factorization properties.

To answer our question we need a microscope with which we can ``just
look at'' the components of $\mathcal{R}^{max}$.  Higgs bundles
provide the tool we need.  A Higgs bundle is a holomorphic bundle
together with a Higgs field, i.e.\ a section of a particular
associated vector bundle.  Such objects appear in the context of
surface group representations as follows. Given a real orientable
surface, say $S$, and any real reductive Lie group, say $G$,
representations of $\pi_1(S)$ in $G$ depend only on the topology of
$S$, i.e.\ on its genus.  Fixing a conformal structure, or
equivalently a complex structure, transforms $S$ into a Riemann
surface (denoted by $X$). This opens the way for holomorphic
techniques and brings in Higgs bundles. The group $G$ appears as the
structure group of the Higgs bundles, which are hence called $G$-Higgs
bundles.  By the non-abelian Hodge theory correspondence
(\cite{hitchin:1987,donaldson:1987,simpson:1988,corlette:1988,
garcia-gothen-mundet:2008}), reductive representations of $\pi_1(X)$
in $G$ correspond to polystable $G$-Higgs bundles, and the
representation variety, i.e.\ the space of conjugacy classes of
reductive representations, corresponds to the moduli space of
polystable Higgs bundles.

Taking $G=\Sp(4,\R)$ we denote the moduli space of polystable
$\Sp(4,\R)$-Higgs bundles by $\mathcal{M}(\Sp(4,\R))$ (or simply $\mathcal{M}$). The non-abelian Hodge
theory correspondence then gives a homeomorphism $\mathcal{M} \simeq
\mathcal{R}(\Sp(4,\R))$. Let $\mathcal{M}^{max} \subset \mathcal{M}$
be subspace corresponding to $\mathcal{R}^{max}$ under this
homeomorphism. If a representation in $\Sp(4,\R)$ factors through a
subgroup, say $G_*\subset\Sp(4,\R)$, then the structure group of the
corresponding $\Sp(4,\R)$-Higgs bundle reduces to $G_*$. Through the
lens of our Higgs bundle microscope, the question we examine thus
becomes: {\it which components of $\mathcal{M}^{max}$ contain
polystable $\Sp(4,\R)$-Higgs bundles for which the structure group
reduces to a subgroup $G_{*}$?}  This is the question we answer.

The geometry of the hermitean symmetric
space $\Sp(4,\R)/\U(2)$ , together with results of Burger, Iozzi and Wienhard
\cite{burger-iozzi-wienhard:2003,burger-iozzi-wienhard:2006}
(see Section \ref{sec:G*}) constrain $G_*$ to be one
of the following three subgroups

\begin{itemize}
\item $G_i=\SL(2,\R)$, embedded via the irreducible representation of
$\SL(2,\R)$ in $\Sp(4,\R)$,
\item $G_p$, the normalizer of the product representation
$$\rho_p: \SL(2,\R)\x\SL(2,\R)\longrightarrow \Sp(4,\R)\ ,$$
\item $G_{\Delta}$, the normalizer of the composition of $\rho_p$ with
the diagonal embedding of $\SL(2,\R)$ in $\SL(2,\R) \x \SL(2,\R)$.
\end{itemize}

For each possible $G_*$ we analyze what $G_*$-Higgs bundles look
like and then, following Feynman's dictum,  we simply check to see which
components of $\mathcal{M}^{max}$ contain Higgs bundles of the
required type. In practice this means that we carefully describe the
structure of maximal $\Sp(4,\R)$-Higgs bundles and compare it to
that of the $G_*$-Higgs bundles.

Our results for each of the possible subgroups are given by Theorems
\ref{cor: FinalTallyGdelta}, \ref{th:GpTally}, and
\ref{thm:GirrTally}. These lead to our main result,
Theorem~\ref{thm:main-result-higgs}, whose essential point is the following.

\begin{theorem}\label{th: intro1}
Of the $3\cdot 2^{2g}+2g-4$ components of $\mathcal{M}^{max}$

(1) $2^{2g}$ are Hitchin components in which the corresponding Higgs
bundles deform to maximal $\SL(2,\R)$-Higgs bundles,

(2) $2\cdot2^{2g}-1$ components have the property that the
corresponding Higgs bundles deform to Higgs bundles which admit a reduction
of structure group to $G_p$, and also deform to ones which admit a reduction
of structure group to $G_{\Delta}$, and

(3) $2g-3$ components have the property that the corresponding Higgs
bundles do not admit a reduction of structure group to a proper
reductive subgroup of $\Sp(4,\R)$.
\end{theorem}

The corresponding result for surface group representations, given in
Theorem~\ref{thm:main-result}, says the following:

\begin{theorem}\label{th: intro2}
Of the $3\cdot 2^{2g}+2g-4$ components of $\mathcal{R}^{max}$

(1) $2^{2g}$ are Hitchin components, i.e.\ the corresponding
representations deform to ones which factor through (Fuchsian) representations into $\SL(2,\R)$,

(2) $2\cdot2^{2g}-1$ components have the property that the
corresponding representations deform to ones which factor through
$G_p$, and also deform to ones which factor through $G_{\Delta}$, and

(3) $2g-3$ components have the property that the corresponding
representations do not factor through any proper reductive subgroup
of $\Sp(4,\R)$.
\end{theorem}

In fact part (1) of  Theorems \ref{th: intro1} and \ref{th: intro2}  follows from Hitchin's  general construction in \cite{hitchin:1992}.   It is nevertheless instructive to see the explicit details of the construction in our particular case, namely $G=\Sp(4,\R)$, and to view the results from a new perspective.  The results about the other maximal components and the other possible subgroups are new.  They raise the interesting problem of gaining a better understanding of the representations which do not deform to representations which factor through a proper reductive subgroup of $\Sp(4,\R)$\footnotemark\footnotetext{ The recent preprint \cite{guichard-wienhard:2009} takes interesting steps in this 
direction.}.

In addition to the main results in Theorems \ref{th: intro1} and \ref{th: intro2}, we also give (in Section \ref{sub:count-explicit}) explicit descriptions of some of the components. Together with the main theorems, these have consequences whose import goes beyond the specific case of $G=\Sp(4,\R)$\footnotemark\footnotetext{We thank an anonymous referee for articulating some of these comments}.  In particular the $2g-3$ components where representations do not factor through any reductive subgroup are remarkable for the following reasons: 

\begin{itemize}
\item the representations in these components all have Zariski dense image in $\Sp(4,\R)$. 
\item the components are smooth but, unlike the Hitchin components, topologically non-trivial.  
\end{itemize}

\noi The group $G=\Sp(4,\R)$ is thus an example of a Lie group with rank greater than $1$ for which the moduli space of  surface group representations into $G$ has components with these properties.  To the best of our knowledge this is the first such example.  Furthermore,  by results of Labourie (\cite{labourie:2005}) and Wienhard (\cite{wienhard:2006}), the mapping class group is known to act properly discontinuously on  $\mathcal{R}^{max}$. The components we describe thus give examples of non-trivial manifolds which carry such actions of the mapping class group.

We note, finally, that the case $G=\Sp(4,\R)$ has features not shared
by $\Sp(2n,\R)$ for $n>2$. In particular, the moduli space of
representations (or Higgs bundles) has anomalously large number of
connected components when $n=2$, compared to the case $n\ge 3$.
Moreover, we prove in Corollary~\ref{cor:n-ge-3} that, when $n\ge 3$,
there are no components of $\mathcal{R}^{max}$ in which all
representations have Zariski dense image.  The case $n=2$ thus demands
treatment as a special case.

\bigskip

\noi {\bf Acknowledgements.}  This paper answers a question first
raised by Bill Goldman at the AIM workshop on Surface Group
Representations in March 2007.  The authors thank the workshop
participants and the AIM staff for making the workshop such a valuable
experience.  The authors thank Marc Burger, Bill Haboush, Nigel
Hitchin, Alessandra Iozzi, Ignasi Mundet, Domingo Toledo, and
especially Bill Goldman and Anna Wienhard, for many useful
conversations and helpful consultations.

\section{Basic background on Higgs bundles and representations}

\subsection{Higgs bundles}
\label{sec:higgs-bundles}

Our main tool for exploring surface group representations is the
relation between such representations and Higgs bundles.  We are
interested primarily in representations in
$\Sp(4,\R)$, but it is useful to state the general definition.

Let $G$ be a \textbf{real reductive Lie group}. Following  Knapp
\cite[p.~384]{knapp:1996}, by this we  mean that we are given the data
$(G,H,\theta,B)$, where $H \subset G$ is a maximal compact subgroup
(cf.\ \cite[Proposition~7.19(a)]{knapp:1996}),  $\theta\colon \lieg
\to \lieg$ is a Cartan involution and $B$ is a
non-degenerate bilinear form on $\lieg$, which is $\Ad(G)$-invariant
and $\theta$-invariant.  The data $(G,H,\theta,B)$ has to satisfy in
addition that
\begin{itemize}
\item
the Lie algebra $\lieg$ of $G$ is reductive
\item
$\theta$ gives  a decomposition  (the Cartan decomposition)
\begin{displaymath}
\lie{g} = \lieh \oplus \liem
\end{displaymath}
into its $\pm1$-eigenspaces, where  $\lieh$ is the Lie
algebras of $H$,
\item
$\lieh$ and $\liem$ are orthogonal under $B$, and $B$ is positive definite
on $\liem$ and negative definite on $\lieh$,

\item
multiplication as a map from $H\times \exp\liem$ into $G$ is an onto
diffeomorphism.
\end{itemize}

We will refer sometimes to the data $(G,H,\theta,B)$, as the 
{\bf Cartan data}.

\begin{remark}
\label{rem:ss-cartan-data}
If $G$ is semisimple, the data $(G,H,\theta,B)$ can be
recovered\footnote{To be precise, the quadratic form $B$ can only
  recovered up to a scalar but this will be sufficient for everything
  we do in this paper.}  from the choice of a maximal compact
subgroup $H \subset G$. There are other situations where less
information does the job, e.g.\ for certain linear groups (see
\cite[p.~385]{knapp:1996}).
\end{remark}

\begin{remark}
The bilinear form $B$ does not play any role in the definition of 
$G$-Higgs bundle that  follows but it
is essential for defining  the  stability condition and  
for making sense of the Hitchin--Kobayashi
correspondence.
\end{remark}

\begin{remark} 
 Note that the compactness of $H$ together with the last property
 above say that $G$ has only finitely many components.
 Note also that we are not assuming, like Knapp, that every
 automorphism $\Ad(g)$ of $\lieg^{\C}$ is inner for every $g\in G$.
\end{remark}

Let $\liegc$ and $\liehc$ be the complexifications of 
$\lieg$ and $\lieh$ respectively, and let 
$H^\C$ be the complexification of $H$. Let 
\begin{equation}\label{eqn:CCartan}
\liegc=\liehc \oplus \liemc
\end{equation}
be the complexification of the Cartan decomposition.  The adjoint
action of $G$ on $\lieg$ restricts to give a representation -- the
isotropy representation -- of $H$ on $\liem$.  Since any two Cartan
decompositions of $G$ are related by a conjugation, the isotropy
representation is independent of the choice of Cartan
decomposition. The same is true of the complexification of this
representation, allowing us to define:

\begin{definition}\label{defn:GHiggs} A {\bf $G$-Higgs bundle} over $X$ is 
a pair $(E,\varphi)$ where
\begin{itemize}
\item $E$ is  a principal holomorphic $H^\C$-bundle $E$ over $X$ and
\item  $\varphi$ is a holomorphic section of $E(\liemc)\otimes K$, where $E(\liemc)$ is the bundle associated to
$E$ via the isotropy representation of $H^\C$ in $\liemc$ and $K$ is the canonical bundle on $X$.
\end{itemize}
\end{definition}

\begin{remark} If $G=\Sp(4,\R)$ then $H=\U(2)$ and $H^\C=\GL(2,\C)$. It is often convenient to replace the
principal $\GL(2,\C)$-bundle in Definition~\ref{defn:GHiggs} with
the vector bundle associated to it by the standard
representation. In the next sections we denote this vector bundle by
$V$.
\end{remark}

In order to define a moduli space of $G$-Higgs bundles 
we need a notion of stability. We briefly recall here the main
definitions (see \cite{garcia-gothen-mundet:2008,garcia-prada-gothen-mundet:2009a} for details).  
Let $\mathfrak{h}^\C_s$ be the semisimple part of $\mathfrak{h}^\C$, that
is, $\mathfrak{h}^\C_s=[\mathfrak{h}^\C,\mathfrak{h}^\C]$.
Choosing a  Cartan subalgebra, let $\Delta$ be  a
fundamental system of roots of $\lieh^\C$. For every subset
$A\subseteq \Delta$ there is a corresponding parabolic 
subalgebra $\mathfrak{p}_A$ of $\mathfrak{h}^\C_s$ and all parabolic
subalgebras can be obtained in this way. Denote by $P_A$ the
corresponding parabolic subgroup of $H^\C$.
Let $\chi$ be an antidominant character of $\liep_A$.
Using the invariant form on $\mathfrak{h}$ defined by $B$,
$\chi$ defines an element $s_\chi\in i\lieh$. 
Now for  $s\in i \mathfrak{h}$, define the sets
\begin{align*}
&\mathfrak{p}_s=\{x\in\liehc\ :\ \Ad(e^{ts})x\text{ is
bounded as }t\to\infty\}\\
&P_s=\{g\in H^\C\ :\ e^{ts}ge^{-ts}\text{ is bounded as }t\to\infty\}\\
&\mathfrak{l}_s=\{x\in\liehc\ :\ [x,s]=0\}\\
&L_s=\{g\in H^\C\ :\ \Ad(g)(s)=s\}.
\end{align*}
One has (see \cite{garcia-prada-gothen-mundet:2009a}) that
for $s\in i \mathfrak{h}$, $\mathfrak{p}_s$ is a parabolic subalgebra
of $\mathfrak{h}^\C$, $P_s$ is a parabolic subgroup of $H^\C$ and the Lie
algebra of $P_s$ is $\mathfrak{p}_s$, $\mathfrak{l}_s$ is a Levi
subalgebra of $\mathfrak{p}_s$ and $L_s$ is a Levi subgroup of $P_s$
with Lie algebra $\mathfrak{l}_s$. Moreover, if $\chi$ is an antidominant
character of $\mathfrak{p}_A$, then
$\mathfrak{p}_A\subseteq\mathfrak{p}_{s_\chi}$ and $L_A\subseteq
L_{s_{\chi}}$ and, if $\chi$ is strictly antidominant,
$\mathfrak{p}_A=\mathfrak{p}_s$ and
$\mathfrak{l}_A=\mathfrak{l}_{s_{\chi}}$.

Let $\iota:H^\C\rightarrow \GL(\liem^\C)$ be
the isotropy representation. We define
\begin{align*}
&\liem_{\chi}^-=\{v\in \liem^\C\ :\ \iota(e^{ts_{\chi}})v
\text{ is bounded as}\;\; t\to\infty\}\\
&\liem^0_{\chi}=\{v\in \liem^\C\ :\ \iota(e^{ts_{\chi}})v=v\;\;
\mbox{for every} \;\; t\}.
\end{align*}
One has  that $\liem^-_{\chi}$ is invariant under the action of
$P_{s_{\chi}}$ and $\liem^0_{\chi}$ is invariant under the action of
$L_{s_{\chi}}$.

Let $E$ be a principal $H^\C$-bundle and $A\subseteq \Delta$. 
Let $\sigma$ denote a reduction of the structure group of $E$ to a
standard parabolic subgroup $P_A$ and let $\chi$ be an
antidominant character of $\plie_A$. Associated to this, 
there is a number called the {\bf degree} of $E$ with respect to
$\sigma$ and $\chi$ that we denote by
$\deg(E)(\sigma,\chi)$. If $\chi$ lifts to a character of $P_A$, 
$\deg(E)(\sigma,\chi)$ is the degree of the line bundle associated 
to $E_\sigma$  via the lift.

A $G$-Higgs bundle  $(E,\varphi)$  is called {\bf semistable}
if for any choice of $P_A, \chi, \sigma$ as above
such that  $\varphi\in H^0(X,E_{\sigma}(\liem_{\chi}^-)\otimes K)$,
we have
$$\deg E(\sigma,\chi)\geq 0.$$
The Higgs bundle
$(E,\varphi)$ is called {\bf stable} if it is semistable and for any $P_A$,
$\chi$ and $\sigma$ as above such that $\varphi\in
H^0(X,E_{\sigma}(\liem_{\chi}^-)\otimes K)$ and $A\not=\emptyset$, $$\deg
E(\sigma,\chi)>0.$$
The Higgs bundle
$(E,\varphi)$ is called {\bf polystable} if it is semistable and for each
$P_A$, $\sigma$ and $\chi$ as in the definition of semistable
$G$-Higgs bundle such that $\deg E(\sigma,\chi)=0$, there exists a
holomorphic reduction of the structure group of $E_{\sigma}$ to the
Levi subgroup $L_A$ of $P_A$,
$\sigma_L\in\Gamma(E_{\sigma}(P_A/L_A))$.  Moreover, in this
case, we require $\varphi\in H^0(X,E(\liem^0_{\chi})\otimes K)$.

We define  the \textbf{moduli space of polystable
$G$-Higgs bundles} $\mathcal{M}(G)$ as the set of isomorphism
classes of polystable $G$-Higgs bundles.
The moduli space $\mathcal{M}(G)$ has the structure of a complex analytic
variety.  This can be seen by the standard slice method (see, e.g.,
Kobayashi \cite{kobayashi:1987}).  Geometric Invariant Theory
constructions are available in the literature for $G$ compact
algebraic (Ramanathan \cite{ramanathan:1975,ramanathan:1996}) 
and for $G$ complex
reductive algebraic (Simpson \cite{simpson:1994,simpson:1995}).  The
case of a real form of a complex reductive algebraic Lie group follows
from the general constructions of Schmitt
\cite{schmitt:2008}. We thus have that
$\mathcal{M}(G)$ is a complex analytic variety, which is
algebraic when $G$ is algebraic.

\subsection{Relation to surface group representations}

Let $G$ be a reductive real Lie group.
By a \textbf{representation} of $\pi_1(X)$ in
$G$ we understand a homomorphism $\rho\colon \pi_1(X) \to G$.
The set of all such homomorphisms,
$\Hom(\pi_1(X),G)$,  is a real analytic  variety, which is algebraic
if $G$ is algebraic.
The group $G$ acts on $\Hom(\pi_1(X),G)$ by conjugation:
\[
(g \cdot \rho)(\gamma) = g \rho(\gamma) g^{-1}
\]
for $g \in G$, $\rho \in \Hom(\pi_1(X),G)$ and
$\gamma\in \pi_1(X)$. If we restrict the action to the subspace
$\Hom^+(\pi_1(X),g)$ consisting of \emph{reductive representations},
the orbit space is Hausdorff.  By a \textbf{reductive representation} we mean
one that, composed with the adjoint representation in the Lie algebra
of $G$, decomposes as a sum of irreducible representations.
If $G$ is algebraic this is equivalent to the Zariski closure of the
image of $\pi_1(X)$ in $G$ being a reductive group.
(When $G$ is compact every representation is reductive.)  The
\emph{moduli space of representations} of $\pi_1(X)$ in $G$
is defined to be the orbit space
\[
\mathcal{R}(G) = \Hom^{+}(\pi_1(X),G) / G. \]

\noi It has the structure of a real 
analytic variety (see e.g.\cite{goldman:1984}) which
is algebraic if $G$ is algebraic and is a complex variety if $G$ is complex.

To see the relation between Higgs bundles and representations
of $\pi_1(X)$, let $h$ be a reduction of structure group of $E_{H^\C}$ from
$H^\C$ to $H$, and let $E_H$ be the principal $H$-bundle defined by
$h$.  Let $d_h$ denote the unique connection on $E_{H^\C}$
compatible with $h$ and let $F_h$ be its curvature. If $\tau$ denotes 
the compact conjugation of $\lieg^\C$ we can formulate the Hitchin equation
$$
F_h -[\varphi,\tau(\varphi)]= 0.
$$
A fundamental result of Higgs bundle theory (see 
\cite{hitchin:1987,simpson:1988,garcia-gothen-mundet:2008})
is that a $G$-Higgs bundle admits a solution  to Hitchin's equation if and only
if the Higgs bundle is polystable.

Now if the Hitchin equation is satisfied then
$$
D=d_h + \varphi - \tau(\varphi)
$$
defines a flat connection on the principal $G$-bundle $E_G=E_H\times_H G$.
The holonomy of this connection thus defines a
representation of $\pi_1(X)$ in $G$.
A fundamental theorem of Corlette \cite{corlette:1988}
(and Donaldson \cite{donaldson:1987} for $G=\SL(2,\C)$; see also Labourie \cite{labourie:1991} for a more general set-up) 
says that
this representations is reductive, and that all  reductive
representations  of $\pi_1(X)$ in $G$ arise in this way.

For semisimple groups the above results establish a homeomorphism  between
isomorphism classes of polystable $G$-Higgs bundles and conjugacy
classes of reductive surface group representations in $G$, i.e.
\begin{equation}
\label{eq:non-ab-hodge}
\mathcal{M}(G)\simeq \mathcal{R}(G).
\end{equation}
It is this homeomorphism that allows us to use Higgs 
bundles to study surface
group representations. If  $G$ is reductive (but
not semisimple) there is a similar correspondence involving
representations of a universal central extension of the fundamental
group.

\subsection{Reduction of structure group}
\label{sec:reduction}

Let $G$ be a real reductive Lie group as defined in
Section~\ref{sec:higgs-bundles}. 
Our main concern is to understand when a surface group representation
in $G$ factors through a subgroup of $G$.  In this section we
reformulate in terms of Higgs bundles what it means for the
representation to factor through a subgroup.

A \textbf{reductive subgroup} of $G$
is a reductive group, say $(G', H', \theta', B')$, such that the
Cartan data is compatible in the obvious sense with the Cartan data of
$(G,H,\theta,B)$ under the inclusion map $G'\hookrightarrow G$. In
particular this implies that $H'\subset H$ and we have a commutative
diagram
\begin{equation}\label{CD:cartandata}
\begin{CD}
\liegc @. = @. \liehc @. \oplus @. \liemc\\
@AAA @.  @AAA @. @AAA\\
\lieg'^\C @. = @. \lieh'^\C @. \oplus @. \liem'^\C.
\end{CD}
\end{equation}
Moreover, the embedding of isotropy representations $\liem'^\C \into
\liem^\C$ is equivariant with respect to the embedding ${H'}^{\C}
\into {H}^{\C}$.

\begin{definition}\label{def:G-reduction}
Let $G$ be a real reductive Lie group and let $G' \subset G$ be a
reductive subgroup. Let $(E,\varphi)$ be a $G$-Higgs bundle. A {\bf
  reduction of $(E,\varphi)$ to a $G'$-Higgs bundle $(E',\varphi')$} is given
by the following data:
\begin{enumerate}
\item A holomorphic reduction of structure group of $E$ to a principal
  ${H'}^{\C}$-bundle $E' \into E$ (equivalently, this is given by a
  holomorphic section $\sigma$ of $E/{H'}^{\C}\to X$).
\item A holomorphic section $\varphi'$ of $E'({\liem'}^{\C})\otimes
  K$ which maps to $\varphi$ under the embedding
  \begin{displaymath}
    E'({\liem'}^{\C})\otimes K \into E({\liem}^{\C})\otimes K.
  \end{displaymath}
\end{enumerate}
\end{definition}

We have the following. 

\begin{proposition}\label{reduction}
Let $G$ be a real reductive Lie group and let $G'\subset G$ be a
reductive subgroup.  Let $(E_{H_{\C}},\varphi)$ be a $G$-Higgs bundle
whose structure group reduces to $G'$. Let $(E_{H'_{\C}},\varphi')$ be
the corresponding $G'$-Higgs bundle.  If $(E_{H_{\C}},\varphi)$ is
polystable as a $G$-Higgs bundle, then $(E_{H'_{\C}},\varphi')$ is
polystable as a $G'$-Higgs bundle.
\end{proposition}

The key fact in the proof of  Proposition \ref{reduction} is that every parabolic 
subgroup of $H'^\C$ and a character
of its Lie algebra extend to a
parabolic subgroup of $H^\C$ and a character of its corresponding Lie algebra.
Moreover,  the corresponding degrees for  parabolic reductions of  structure
group of the bundles coincide. This can be seen using filtrations
of the vector bundles associated to $E_{H^{\C}}$ and  $E_{H'^{\C}}$ 
via an auxiliary representations of $H^\C$
(see \cite{garcia-prada-gothen-mundet:2009a}).

In the situation of Proposition~\ref{reduction}, the non-abelian Hodge
theory correspondence implies that the polystable $G'$-Higgs bundles
obtained from polystable $G$-Higgs bundles correspond to
$G'$-representations of $\pi_1(X)$. Conversely, let $\rho$ be a
reductive surface group representation in $G$ which factors through a
reductive subgroup $G'$. Then it is clear that the corresponding
polystable $G'$-Higgs bundle is a $G'$-reduction of the $G$-Higgs
bundle corresponding to $\rho$. Thus Proposition~\ref{reduction} has
the following immediate corollary.

\begin{proposition}
Let $G$ be a real reductive Lie group and let $G'\subset G$ be a
reductive subgroup. 

(1) A reductive $\pi_1(X)$-representation in $G$ factors through a
reductive representation in $G'$ if and only if the corresponding
polystable $G$-Higgs bundle admits a reduction of structure group to
$G'$.

(2) Let $\rho:\pi_1(X)\longrightarrow G$ be a reductive representation
and let $(E_{H^{\C}},\phi)$ be the corresponding polystable
$G$-Higgs bundle. Suppose that $(E_{H^{\C}},\phi)$ defines a point
in a connected component $\mathcal{M}_c(G)\subset \mathcal{M}(G)$. The
representation $\rho$ deforms to a representation which factors
through $G'$ if and only if $\mathcal{M}_c(G)$ contains a point
represented by a $G$-Higgs bundle that admits a reduction of
structure group to $G'$.
\end{proposition}

Let $G$ be a real reductive Lie group and let $G' \into G$ be an
embedding of the Lie group $G'$ as a closed subgroup. One may ask
whether the Cartan data of $G$ induces Cartan data on $G'$ such that
$G'$ is a reductive subgroup of $G$. In the following we answer this
question.

\begin{definition}
An embedding of Lie algebras $\lieg' \subset \lieg$ is {\bf
  canonical} with respect to a Cartan involution, $\theta$, on
$\lieg$ if $\theta(\lieg') = \lieg'$.
\end{definition}

\begin{lemma}\label{rem:subgroups} 
 Let $G' \subset G$ be a closed Lie subgroup such that $\lieg'
 \subset \lieg$ is canonically embedded. Then $H' = H \cap G'$ is a
 maximal compact subgroup of $G'$. Moreover, if we let $\theta'$ and
 $B'$ be the restrictions of $\theta$ and $B$, respectively, to
 $\lieg'$, then $(G',H',\theta',B')$ is a reductive subgroup of
 $(G,H,\theta,B)$.
\end{lemma}

In view of this Lemma, we make the following convention.

\begin{convention*}
 Whenever $G' \subset G$ is a closed subgroup whose Lie algebra is
 canonically embedded, we consider $G'$ as a reductive subgroup of
 $G$ with the induced Cartan data.
\end{convention*}

\begin{remark}
 \label{rem:ss-cartan-compatibility}
 If, in the situation of Lemma~\ref{rem:subgroups}, $G'$ is
 semisimple, the structure of reductive subgroup induced from $G$
 must coincide with the one coming from the choice of the maximal
 compact subgroup $H' = H \cap G'$ (cf.\
 Remark~\ref{rem:ss-cartan-data}).

 Thus, if we are given a semisimple closed subgroup $G' \subset G$
 with an a priori choice of maximal compact $H' \subset G'$, then in
 order to check that the corresponding Cartan data coincides with the
 Cartan data induced from $G$, it suffices to check that $H' = H
 \cap G$ and that $\lieg' \into \lieg$ is canonically embedded.
\end{remark}

\section{$\Sp(4,\R)$-Higgs bundles}

\subsection{Definition of $\Sp(4,\R)$ and choice of Cartan data}\label{subs:Sp4cartan}

The Lie group $\Sp(4,\R)$ is the subgroup of $\SL(4,\R)$ which
preserves a symplectic form on $\R^4$. The description of the group
depends on the choice of symplectic form. We use the following
conventions.

\begin{definition} Let
\begin{equation}\label{defn:J13}
J_{13}=\begin{pmatrix}0 & I_2 \\-I_2 & 0\end{pmatrix}
\end{equation}
\noindent where $I_2$ is the $2\x2$ identity matrix.  This defines the
symplectic form $\omega_{13}(a,b)=a^tJ_{13}b$ where $a$ and $b$ are vectors in $\R^4$, i.e. 
\begin{equation}
\omega_{13}=x_1\wedge x_3+x_2\wedge x_4.
\end{equation}
The {\bf symplectic group} in dimension four, defined using $J_{13}$, is thus
\begin{equation}\label{defn: Sp4J13}
\Sp(4,\R)=\{g\in\SL(4,\R) \suchthat g^tJ_{13}g=J_{13}\ \}.
\end{equation}
\end{definition}

\begin{remark} Later on (see Sections \ref{subs: subgroups}, \ref{subs:irep}) it will be convenient to consider other choices of symplectic form (denoted by $J_{12}$ and $J_0$).  The resulting changes in description will be pointed out as needed.
\end{remark}

\noi The maximal compact subgroups of $\Sp(4,\R)$ are isomorphic to $\U(2)$,
i.e.\ in the notation of the previous section, if $G=\Sp(4,\R)$ then
$H=\U(2)$.  Using symplectic form $J_{13}$, we fix the subgroup $\U(2)\subset \Sp(4,\R)$ given by
\begin{equation}\label{eqn:U(2)}
\U(2)=\Big\{\ \begin{pmatrix}A & B \\-B & A \end{pmatrix} \suchthat 
A^tA+B^tB=I\ ,\ A^tB-B^tA=0\ \Big\}\ ,
\end{equation}
i.e.\ given by the embedding
\begin{equation}\label{eqn:U2J13}
A+iB\mapsto\begin{pmatrix}A & B \\-B & A \end{pmatrix}.
\end{equation}
It follows from (\ref{defn: Sp4J13}) and (\ref{eqn:U2J13}) that the
Cartan decomposition corresponding to our choice of $\U(2)$ is
defined by the involution
\begin{equation}\label{eqn:theta}
\theta(X)=-X^t
\end{equation}

\noi on 
\begin{equation}\label{eqn:sp4}
\mathfrak{sp}(4,\R)=\Big\{\begin{pmatrix} A&B\\C&-A^t\end{pmatrix} \suchthat A,B,C\in \Mat_2(\R)\ ;\
B^t=B\ ,\ C^t=C\ \Big\}\ .
\end{equation}

\noi This gives
\begin{equation}\label{eqn:CartanR}
\mathfrak{sp}(4,\R)=\mathfrak{u}(2)\oplus\mathfrak{m}
\end{equation}
with
\begin{align}\label{eqn:u2}
\mathfrak{u}(2)=&\Big\{\begin{pmatrix} A&B\\-B&A\end{pmatrix} \suchthat A,B\in \Mat_2(\R)\ ;\
A^t=-A\ ,B^t=B\ \Big\}\ ,\\
\mathfrak{m}=&\Big\{\begin{pmatrix} A&B\\B&-A\end{pmatrix} \suchthat  A,B\in \Mat_2(\R)\ ;\
A^t=A\ ,B^t=B\ \Big\}\ .
\end{align}
The complexification of (\ref{eqn:CartanR}),
\begin{equation}\label{eqn:CartanC}
\mathfrak{sp}(4,\C)=\mathfrak{gl}(2,\C)\oplus\liemc
\end{equation}
is obtained by replacing $\Mat_2(\R)$ with $\Mat_2(\C)$. In particular, we identify $\mathfrak{gl}(2,\C)$ via\footnotemark\footnotetext{This corresponds to mapping
\begin{equation*}
Z\mapsto\begin{pmatrix}\frac{Z-Z^t}{2}&\frac{Z+Z^t}{2i}\\-\frac{Z+Z^t}{2i}&\frac{Z-Z^t}{2}\end{pmatrix}.
\end{equation*}
}
\begin{equation}
\mathfrak{gl}(2,\C)=\{\begin{pmatrix} A&B\\-B&A\end{pmatrix}
\suchthat A,B\in 
\Mat_2(\C)\ ;\
A^t=-A\ ,B^t=B\ \}
\end{equation}

\noi Notice that after  conjugation by $T=\begin{pmatrix}I& iI \\
I & -iI
\end{pmatrix}$ , i.e.\ after the change of basis (on $\C^4$) effected by $T$, we
identify the summands in the Cartan decomposition of
$\lie{sp}(4,\C)\subset\lie{sl}(4,\C)$ as
\begin{align}\label{eqn:gl2CInsp4C.2}
\mathfrak{gl}(2,\C) =& \Big\{\begin{pmatrix} Z&0\\0&-{Z^t}\end{pmatrix}\
|\  Z \in 
\Mat_2(\C) \Big\}\ ,\nonumber\\
\liemc =&\Big \{\begin{pmatrix} 0&\beta\\\gamma&0\end{pmatrix}\ |\  \beta,\gamma\in \Mat_2(\C), \beta^t=\beta\ ,\ \gamma^t=\gamma\ \Big\}\nonumber\\
=& \Sym^2(\C^2)\oplus \Sym^2((\C^2)^*)\ .
\end{align}

\noi This corresponds to an embedding of $\U(2)$ (the maximal compact
subgroup of $ \Sp(4,\R)$) in $\SU(4)$ (the maximal compact subgroup in
$\SL(4,\C)$) given by
\begin{equation}\label{eqn:U2inGL4C}
U\mapsto\begin{pmatrix}U & 0 \\0 & (U^t)^{-1}
\end{pmatrix}\ \mathrm{where}\ U^*U=I\ .
\end{equation}

\subsection{Definition of $\Sp(4,\R)$-Higgs bundles.}

We fix $G=\Sp(4,\R)$ and $H=\U(2)$ as in Section \ref{subs:Sp4cartan}.
Given a holomorphic principal $\GL(2,\C)$-bundle on $X$, say $E$, let
$V$ denote the rank 2 vector bundle associated to $E$ by the standard
representation.  The Cartan decomposition described in Section
\ref{subs:Sp4cartan} shows (see (\ref{eqn:gl2CInsp4C.2})) that we can
identify
\begin{equation}\label{eqn:Higgsfields}
E(\liemc)=\Sym^2(V)\oplus \Sym^2(V^*).
\end{equation}
Definition (\ref{defn:GHiggs}) thus specializes to the
following:

\begin{definition}\label{defn:SpHiggs} With $G=\Sp(4,\R)$ and $H=\U(2)$ as in Section \ref{subs:Sp4cartan}, an 
{\bf $\Sp(4,\R)$-Higgs bundle}
over $X$ is defined by a triple
$(V,\beta,\gamma)$ consisting of a rank $2$ holomorphic vector
bundles $V$ and symmetric homomorphisms
$$
\beta: V^*\lra V \otimes K \;\;\;\mbox{and}\;\;\; \gamma: V\lra V^*
\otimes K.
$$
\end{definition}

\noi  Except when it is important to keep track of the maximal
compact subgroup, we will refer to these objects as $\Sp(4,\R)$-Higgs
bundles.  The composite embedding
\begin{equation}
\Sp(4,\R)\hookrightarrow\Sp(4,\C)\hookrightarrow\SL(4,\C)
\end{equation}
allows us to reinterpret the defining data for $\Sp(4,\R)$-Higgs
bundles as data for special $\SL(4,\C)$-Higgs bundles (in the original
sense of \cite{hitchin:1992}).  Indeed, the embeddings (\ref
{eqn:gl2CInsp4C.2}) show 
that the triple $(V,\beta,\gamma)$ in
Definition \ref{defn:SpHiggs} is equivalent to the pair
$(\mathcal{E},\varphi)$, where
\goodbreak
\begin{enumerate}
\item $\mathcal{E}$ is the rank $4$ holomorphic bundle $\mathcal{E}=V\oplus V^*$, and
\item $\varphi$ is a Higgs field $\varphi:\mathcal{E}\longrightarrow \mathcal{E}\otimes K$
given by $
\varphi=\left(
\begin{smallmatrix}
0 & \beta\\
\gamma& 0
\end{smallmatrix}\right)$.
\end{enumerate}

\begin{remark}
\label{rem:SL2R-higgs}
The definition of $\Sp(2n,\R)$-Higgs bundles for general $n$ is of
course entirely analogous and later we shall need the special case
$n=1$, corresponding to $G=\Sp(2,\R) = \SL(2,\R)$. Thus an
$\SL(2,\R)$-Higgs bundle is given by the data $(L,\beta,\gamma)$,
where $L$ is a line bundle, $\beta \in H^0(L^2 K)$ and $\gamma \in
H^0(L^{-2}K)$.
\end{remark}

\subsection{Stability}

The general definition of (semi-)stability for
$G$-Higgs bundles given in Section~\ref{sec:higgs-bundles} 
simplifies in the case $G=\Sp(2n,\R)$ (see \cite[Section~3]{garcia-gothen-mundet:2008} or  \cite{schmitt:2008}).  To state the simplified stability condition, we use the following
notation. For any line subbundle $L\subset V$ we denote by $L^{\perp}$ the
subbundle of $V^*$ in the kernel of the projection onto $L^*$, i.e.
\begin{equation}\label{eqn: Lperp}
0\longrightarrow L^{\perp}\longrightarrow V^*\longrightarrow
L^*\longrightarrow 0\ .
\end{equation}
Moreover, for subbundles $L_1$ and $L_2$ of a vector bundle $V$, we
denote by $L_1 \otimes_S L_2$ the symmetrized tensor product, i.e.\
the symmetric part of $L_1 \otimes L_2$ inside the symmetric product
$S^2V$ (these bundles can be constructed in standard fashion from the
corresponding representations, using principal bundles).  For $n=2$, i.e.\ for $G=\Sp(4,\R)$, the stability condition then takes
the following form.

\begin{proposition}\label{prop: n=2stability}
An $\Sp(4,\R)$-Higgs bundle $(V,\beta,\gamma)$ is semistable if and
only if all the following conditions hold
\begin{enumerate}
\item\label{item:1} If $\beta=0$ then $\deg(V)\geq 0$.
\item\label{item:2} If $\gamma=0$ then $\deg(V) \leq 0$.
\item\label{item:line} Let $L \subset V$ be a line subbundle.
\begin{enumerate}
\item\label{item:line1}
If $\beta \in H^0(L \otimes_S V \otimes K)$ and $\gamma \in
H^0(L^\perp\otimes_S V^*\otimes K)$ then
\begin{math}
\deg(L) \leq \frac{\deg(V)}{2}.
\end{math}
\item\label{item:line2}
If $\gamma \in H^0((L^\perp)^2\otimes K)$ then
\begin{math}
\deg(L) \leq 0.
\end{math}
\item\label{item:line3}
If $\beta \in H^0(L^2\otimes K)$ then
\begin{math}
\deg(L) \leq \deg(V).
\end{math}
\end{enumerate}
\end{enumerate}
If, additionally, strict inequalities hold in (\ref{item:line}), then
$(V,\beta,\gamma)$ is stable.
\end{proposition}

Similarly, the notion of polystability simplifies as follows.

\begin{proposition}
\label{prop:polystability}
Let $(V,\beta,\gamma)$ be an $\Sp(4,\R)$-Higgs bundle with
$\deg(V)\neq 0$. Then $(V,\beta,\gamma)$ is polystable if it is
either stable, or if there is a decomposition $V= L_1 \oplus L_2$ of
$V$ as a direct sum of line bundles, such that one of the following
conditions is satisfied:
\begin{enumerate}
\item The Higgs fields satisfy $\beta=\beta_1 + \beta_2$ and
$\gamma=\gamma_1 + \gamma_2$, where
\begin{displaymath}
\beta_i \in H^0(L_i^2\otimes K)\quad\text{and}\quad
\gamma_i \in H^0(L_i^{-2}\otimes K)
\end{displaymath}
for $i=1,2$. Furthermore, the $\Sp(2,\R)$-Higgs bundles
$(L_i,\beta_i,\gamma_i)$ are stable for $i=1,2$ and there is an
isomorphism of $\Sp(2,\R)$-Higgs bundles $(L_1,\beta_1,\gamma_1)
\simeq (L_2,\beta_2,\gamma_2)$.
\item The Higgs fields satisfy 
$$\begin{cases}\beta \in H^0((L_1L_2\oplus L_2L_1)\otimes K)\\ \gamma \in H^0((L_1^{-1}L_2^{-1}\oplus L_2^{-1}L_1^{-1})\otimes K)\end{cases}\  .$$
\noi Furthermore, $\deg(L_1) = \deg(L_2) = \deg(V)/2$
and the rank $2$ Higgs bundle
$(L_1 \oplus L_2^*,
\left(
\begin{smallmatrix}
0 & \beta \\
\gamma & 0
\end{smallmatrix}
\right))$
is stable.
\end{enumerate}
\end{proposition}

\begin{remark}\label{rem:stable-not-simple}

If $\deg V=0$ then  there are other possible decompositions
for a polystable $\Sp(4,\R)$-Higgs bundle; and
if  $(V,\beta,\gamma)$ is as in (1) of
Proposition~\ref{prop:polystability} but with $(L_1,\beta_1,\gamma_1)$
and $(L_2,\beta_2,\gamma_2)$ non isomorphic then it is a stable
$\Sp(4,\R)$-Higgs bundle which is not simple (see
Theorem 3.40 in \cite{garcia-gothen-mundet:2008} for details).
\end{remark}

The following result \cite{garcia-gothen-mundet:2008}
relating polystability of $\Sp(4,\R)$-Higgs bundles to
polystability of $\GL(4,\C)$-Higgs bundles is useful.  It is
important to point out that, though the polystability conditions
coincide, the stability condition for a $\Sp(4,\R)$-Higgs bundle is
weaker than the stability condition for the corresponding
$\GL(4,\C)$-Higgs bundle.

\begin{proposition}[\protect{\cite[Theorem
5.13]{garcia-gothen-mundet:2008}}]
\label{prop:Sp4R-GL4C-polystability-equivalence}
An $\Sp(4,\R)$-Higgs bundle $(V,\beta,\gamma)$ is polystable if
and only if the $\GL(4,\C)$-Higgs bundle $(V \oplus V^*, \varphi = \left(
\begin{smallmatrix}
0 & \beta \\
\gamma & 0
\end{smallmatrix}
\right))$ is polystable.
\end{proposition}
Recall that a $\GL(4,\C)$-Higgs bundle $(\mathcal{E},\varphi)$ is
stable if, for any proper non-zero $\varphi$-invariant subbundle $F
\subseteq \mathcal{E}$ satisfies $\mu(F) < \mu(\mathcal{E})$, where
$\mu(F) = \deg(F)/\rk(F)$ is the slope of the subbundle. The Higgs
bundle $(\mathcal{E},\varphi)$ is polystable if it is the direct sum
of stable Higgs bundles, all of the same slope.
Moreover, to check that the $\GL(4,\C)$-Higgs bundle
$(\mathcal{E} = V \oplus V^*, \varphi = \left(
\begin{smallmatrix}
0 & \beta \\
\gamma & 0
\end{smallmatrix}
\right))$ is stable, it suffices to consider $\varphi$-invariant
subbundles which respect the decomposition $\mathcal{E} = V \oplus
V^*$ (see \cite{bradlow-garcia-prada-gothen:2003}).

\begin{remark}
\label{rem:SL2R-stability}
Similarly, the stability condition for an $\SL(2,\R)$-Higgs bundle
$(L,\beta,\gamma)$ simplifies as follows.
\begin{enumerate}
\item If $\deg(L)>0$ then $(L,\beta,\gamma)$ is stable if and only if
$\gamma \neq 0$.
\item If $\deg(L)<0$ then $(L,\beta,\gamma)$ is stable if and only if
$\beta \neq 0$.
\item If $\deg(L)=0$ then $(L,\beta,\gamma)$ is polystable if and only if
either $\beta = 0 = \gamma$ or both $\beta$ and $\gamma$ are nonzero.
\end{enumerate}
Moreover, if $\deg(L) \neq 0$, then stability, polystability and
semistability are equivalent conditions. Notice that from the semistability
condition if $\deg(L) >0$, since $\gamma\neq 0$, we must have that 
$\deg (L) \leq g-1$; and similarly, if  
$\deg (L) <0$, since $\beta\neq 0$,  we must have that 
$\deg (L) \geq  1-g$. We thus  have  the Milnor--Wood inequality for 
$\SL(2,\R)$-Higgs bundles (see \cite{milnor:1957,goldman:1988,hitchin:1987}).

Finally, in a manner analogous to
Proposition~\ref{prop:Sp4R-GL4C-polystability-equivalence}, we have
that $(L,\beta,\gamma)$ is a polystable $\SL(2,\R)$-Higgs bundle if
and only if
$$(L \oplus L^{-1},\left(
\begin{smallmatrix}
 0 & \beta \\
 \gamma & 0
\end{smallmatrix}
\right))$$
is a polystable $\SL(2,\C)$-Higgs bundle.
\end{remark}

\subsection{Toledo invariant and moduli spaces}

The basic topological invariant of an $\Sp(4,\R)$-Higgs bundle is
the degree of $V$.
\begin{definition}
\label{def:toledo}
The {\bf Toledo invariant} of the $\Sp(4,\R)$-Higgs bundle
$(V,\gamma,\beta)$ is the integer
\begin{displaymath}
d = \deg(V).
\end{displaymath}
\end{definition}
From the point of view of representations of the fundamental group,
the Toledo invariant is defined for representations into any group
$G$ of hermitean type. This justifies the terminology used in the
definition. 

The following inequality for the Toledo invariant has a
long history, going back to Milnor \cite{milnor:1957}, Wood
\cite{wood:1971}, Dupont \cite{dupont:1978}, Turaev
\cite{turaev:1984}, Domic--Toledo \cite{domic-toledo:1987} and
Clerc--\O{}rsted \cite{clerc-orsted:2003}. It is
usually known as the \emph{Milnor--Wood inequality}.

\begin{proposition}
\label{prop:MW}
Let $(V,\beta,\gamma)$ be a semistable $\Sp(4,\R)$-Higgs
bundle. Then
\begin{displaymath}
\abs{d} \leq 2g-2.
\end{displaymath}
\qed
\end{proposition}
The sharp bound for $G = \Sp(4,\R)$ was given by Turaev.
In its most general form the Milnor--Wood inequality has been
proved by Burger, Iozzi and Wienhard. For a proof in the present
context of Higgs bundle theory, see \cite{gothen:2001}.

We call $\Sp(4,\R)$-Higgs bundles with Toledo invariant $d = 2g-2$
\textbf{maximal}, and define \textbf{maximal representations} $\rho\colon
\pi_1(X) \to \Sp(4,\R)$ similarly.

For simplicity, we shall henceforth use the notation
\begin{displaymath}
\mathcal{M}_d = \mathcal{M}_d(\Sp(4,\R))
\end{displaymath}
for the moduli space parametrizing isomorphism classes of polystable
$\Sp(4,\R)$-Higgs bundles $(V,\beta,\gamma)$ with $\deg(V) = d$. We
will denote the components with maximal positive Toledo invariant by
$\mathcal{M}^{max}$, i.e.
\begin{displaymath}
\mathcal{M}^{max} = \mathcal{M}_{2g-2} \ .
\end{displaymath}
We remark (cf.\ \cite{garcia-gothen-mundet:2008}) that there is an
isomorphism $\mathcal{M}_{d} \simeq \mathcal{M}_{-d}$, given by the map
$(V,\beta,\gamma) \to (V^*,\gamma,\beta)$. This justifies restricting
attention to the case $d\geq 0$ of positive Toledo invariant .

\subsection{Maximal $\Sp(4,\R)$-Higgs bundles and Cayley partners}
\label{subs:cayley}

The Higgs bundle proof \cite{gothen:2001} of Proposition \ref{prop:MW}
has the following important consequence.

\begin{proposition}\label{prop:maxtoledo}
Let $(V,\beta,\gamma)$ be a polystable $\Sp(4,\R)$-Higgs bundle. If
$\deg V=2g-2$, i.e.\ if $d$ is maximal and positive, then
$$\gamma:V\longrightarrow V^*\otimes K$$
is an isomorphism.
\end{proposition}

\noi If $\gamma:V\longrightarrow V^*\otimes K$ is an isomorphism,
then some of the conditions in Proposition \ref{prop: n=2stability}
cannot occur.  The  stability condition then reduces to:

\begin{proposition}\label{prop: n=2stability-I}
Let $(V,\beta,\gamma)$ be an $\Sp(4,\R)$-Higgs bundle and assume that
$\gamma\colon V \to V^* \otimes K$ is an isomorphism. Set
\begin{equation}
\tilde{\beta}=(\beta\otimes 1)\circ \gamma\colon V\to V\otimes K^2.
\end{equation}
\noindent Then $(V,\beta,\gamma)$ is semi-stable if and only if for
any line subbundle $L \subset V$  isotropic with respect to $\gamma$
and such that $\tilde{\beta}(L) \subseteq  L\otimes K^2$, the
following condition is satisfied
\begin{displaymath}
\mu(L) \le \mu(V)\ .
\end{displaymath}
If strict inequality holds then $(V,\beta,\gamma)$ is stable.
\end{proposition}

If we fix a square root of $K$, i.e.\ if we pick a line bundle
$L_0$ such that $L_0^2=K$, and define
\begin{equation}\label{defn:CayleyW}
W=V^*\otimes L_0
\end{equation}
then it follows from Proposition \ref{prop:maxtoledo} that
the map
\begin{equation}\label{defn:qW}
q_W:=\gamma\otimes I_{L_0^{-1}}: W^* \ra W
\end{equation}
defines a symmetric, non-degenerate form on $W$, i.e.\
$(W,q_W)$ is an $\OO(2,\C)$-holomorphic bundle.  The remaining part
of the Higgs field, i.e.\ the map $\beta$ defines a $K^2$-twisted
endomorphism
\begin{align}\label{defn:theta}
\theta = (\gamma\otimes I_{K\otimes L_0})&\circ (\beta \otimes
I_{L_0})\colon W \ra W\otimes K^2\ .
\end{align}
The map $\theta$ is $q_W$-symmetric, i.e.\ it takes values in the
isotropy representation for $\GL(2,\R)$.  The pair $(W,q_W,\theta)$
thus satisfies the definition of a $G$-Higgs bundle with
$G=\GL(2,\R)$, except for the fact that the Higgs field $\theta$
takes values in $E(\liemc)\otimes K^2$ instead of in
$E(\liemc)\otimes K$.  We say that $(W,\theta)$ defines a {\bf
$K^2$-twisted Higgs pair with structure group $\GL(2,\R)$} (see
\cite{garcia-gothen-mundet:2008} for more details).

\begin{definition}\label{defn:Cayley}
We call $(W,q_W,\theta)$ the {\bf Cayley partner} of the $\Sp(4,\R)$-Higgs
bundle $(V,\beta,\gamma)$.
\end{definition}
The original $\Sp(4,\R)$-Higgs bundle can clearly be recovered from
the defining data for its Cayley partner. We refer to
\cite{GDP} for more details on this construction, including an
exposition of the general framework which justifies our
terminology. Occasionally, when the section $\theta$ is not directly
relevant for our considerations, we shall also refer to the orthogonal
bundle $(W,q_W)$ as the Cayley partner of $(V,\beta,\gamma)$.

The following Proposition sums up the essential point of the
constructions of this section.
\begin{proposition}\label{prop:maxtoledo2} Let $(V,\beta,\gamma)$ be a
polystable $\Sp(4,\R)$-Higgs bundle with maximal positive Toledo
invariant, i.e.\ with $\deg(V)=2g-2$. Then $V$ can be written as
\begin{equation}\label{eqn:V=WL}
V=W\otimes L_0\ ,
\end{equation}
where $W$ is an $\OO(2,\C)$-bundle and $L_0$ is a
line bundle such that
\begin{equation}
L_0^2=K\ .
\end{equation}
Also, the isomorphism $\gamma$ is given by
\begin{equation}
\gamma=q\otimes I_{L_0}:W\otimes L_0\longrightarrow W^*\otimes L_0\
,
\end{equation}
where $q$ defines the orthogonal structure on $W$ and
$I_{L_0}$ is the identity map on $L_0$, and
\begin{equation}
\det(V)^2=K^2\ .
\end{equation}
\end{proposition}

\subsection{Connected components of the moduli space}\label{sub:count}

The moduli space $\mathcal{M}^{max}$ of maximal $\Sp(4,\R)$-Higgs
bundles is not connected. Its connected components of
$\mathcal{M}^{max}$ were determined in \cite{gothen:2001}. In
contrast, each moduli space $\mathcal{M}_d$ for $\abs{d} < 2g-2$ is
connected (see \cite{garcia-prada-mundet:2004}). In this
section we explain the count of components of $\mathcal{M}^{max}$ 
and identify the Higgs bundles appearing in each
component. 

The key to the count of the components of $\mathcal{M}^{max}$ is
Proposition~\ref{prop:maxtoledo}. The fact that the orthogonal bundle
$(W,q_W)$ underlying the Cayley partner is an $\OO(2,\C)$-bundle
reveals new topological invariants, namely the first and
second Stiefel--Whitney classes
\begin{align}
w_1(W,q_W)&\in H^1(X;\Z/2) \simeq \Z/2^{2g} \\
w_2(W,q_W)&\in H^2(X;\Z/2) \simeq \Z/2\ .
\end{align}

Rank 2 orthogonal bundles were classified by Mumford in \cite{mumford}
(though the reducible case (3) was omitted there):

\begin{proposition}\label{prop:O(2)bundles}
A rank 2 orthogonal bundle $(W,q_W)$ is one of
the following:
\begin{enumerate}
\item $W=L\oplus L^{-1}$, where $L$ is a line bundle on $X$, and
$q_W=\left(
\begin{smallmatrix}
0 & 1 \\
1 & 0
\end{smallmatrix}\right)$. In this case $w_1(W,q_W) = 0$.
\item $W=\pi_*(\tilde{L}\otimes \iota^*\tilde{L}^{-1})$ where
$\pi:\tilde{X}\longrightarrow X$ is a connected double cover, $\tilde{L}$ is a
line bundle on $\tilde{X}$, and $\iota:\tilde{X}\longrightarrow
\tilde{X}$ is the covering involution. The quadratic form is locally
of the form $q_W=\left(
\begin{smallmatrix}
0 & 1 \\
1 & 0
\end{smallmatrix}\right)$. In this case $w_1(W,q_W) \in
H^1(X;\Z/2) $ is the non-zero element defining the double cover.
\item $W=L_1\oplus L_2$ where $L_1$ and $L_2$ are line bundles on
$X$ satisfying $L_i^2=\cO_X$, and $q_W=q_1+q_2$ where $q_i$ defines
the isomorphism $L_i\simeq L_i^{-1}$. In this case $w_1(W,q_W) =
w_1(L_1,q_1) + w_1(L_2,q_2)$.
\end{enumerate}
\end{proposition}
Note that cases (1) and (3) above are not mutually exclusive: they
coincide when $V = L \oplus L$ with $L^2 = \mathcal{O}$ and
$q_W=\left(
\begin{smallmatrix}
1 & 0 \\
0 & 1
\end{smallmatrix}\right)$.

\begin{remark}
In case (2) above, the line bundles of the form $M=\tilde{L}\otimes
\iota^*\tilde{L}^{-1}$ constitute the kernel of $1+\iota^*\colon
\Jac(\tilde X) \to \Jac(\tilde X)$. Moreover, this kernel
consists of two components $P^+$ and $P^-$ (distinguished by the
degree of $\tilde L$ modulo $2$), each one of them a translate of
the Prym variety of the cover (cf.\ \cite{mumford}).  It can be
shown that the value of $w_2(W,q_W)$ is $0$ or $1$ depending on
whether $M$ belongs to $P^+$ or $P^-$ (see
\cite[Proposition~5.14]{gothen:2001}).
\end{remark}

Recall that the first Stiefel--Whitney class is the obstruction to the
existence of a reduction of structure group to $\SO(2,\C) \subset
\OO(2,\C)$. Thus, with $\SO(2,\C) \simeq \C^*$ via
$\lambda\mapsto 
\left(
\begin{smallmatrix}
\lambda & 0 \\
0 & \lambda^{-1}
\end{smallmatrix}\right)$,
we get:
\begin{proposition}
\label{prop:def_Chern}
Let $(W,q_W)$ be an $\OO(2,\C)$-bundle. Then $w_1(W,q_W) $ equals zero if and
only if $(W,q_W)$ is of the kind described in (1) of
Proposition~\ref{prop:O(2)bundles}. In this case the second
Stiefel--Whitney class, $w_2(W,q_W)$, lifts to the integer class
$c_1(L) \in H^2(X;\Z)$.
\end{proposition}

Let $(V,\beta,\gamma)$ be a maximal semistable $\Sp(4,\R)$-Higgs
bundle and let $(W,q_W)$ be defined by (\ref{defn:CayleyW}) and
(\ref{defn:qW}).  We define topological invariants of
$(V,\beta,\gamma)$ as follows:
\begin{displaymath}
w_i(V,\beta,\gamma) = w_i(W,q_W), \quad i=1,2.
\end{displaymath}
Note that these invariants are well defined because the
Stiefel--Whitney classes are independent of the choice of the square
root $L_0$ of the canonical bundle used to define the Cayley partner
$(W,q_W)$. When $w_1(V,\beta,\gamma) = 0$, the class
$w_2(V,\beta,\gamma)$ lifts to the integer invariant $\deg(L)$, where
$W = L \oplus L^{-1} = V \otimes L_0^{-1}$ is the
vector bundle underlying the Cayley partner of $(V,\beta,\gamma)$.

\begin{proposition}
\label{prop:def_N}
Let $(V,\beta,\gamma)$ be a maximal semistable $\Sp(4,\R)$-Higgs
bundle with $w_1(V,\beta,\gamma) = 0$ and let $(W = L \oplus L^{-1},
q_W = \left(
\begin{smallmatrix}
0 & 1 \\
1 & 0
\end{smallmatrix}\right))$
be its Cayley partner. Then there is a line bundle
$N$ such that
$$V=N\oplus N^{-1}K,$$
and, with respect to this decomposition,
$$\gamma=\left(
\begin{smallmatrix}
0 & 1 \\
1 & 0
\end{smallmatrix}\right)\in H^0(S^2V^*\otimes K)
\quad\text{and}\quad
\beta=\left(
\begin{smallmatrix}
\beta_1 & \beta_3 \\
\beta_3 & \beta_2
\end{smallmatrix}\right)\in H^0(S^2V\otimes K).$$
The degree of $N$ is given by
\begin{displaymath}
\deg(N) = \deg(L) + g-1.
\end{displaymath}
Moreover,
\begin{displaymath}
0 \leq \deg(L) \leq 2g-2
\end{displaymath}
and, for $\deg(L) > 0$,
\begin{displaymath}
\beta_2 \neq 0.
\end{displaymath}
When $\deg(L) = 2g-2$ the line bundle $N$ satisfies
\begin{equation}
\label{eq:N2K3}
N^2 = K^3.
\end{equation}
\end{proposition}
\begin{proof}
The statement about the shape of $(V,\beta,\gamma)$ follows by
applying Propositions~\ref{prop:O(2)bundles} and
\ref{prop:def_Chern} to the Cayley partner, letting $N = L L_0$.

Assuming without loss of generality that $\deg(L) \geq 0$, the fact
that $0 \neq \beta_2 \in H^0(X,N^{-2}K^3)$ follows easily from
semistability (cf.\ \cite{gothen:2001}). The rest now follows from
$\deg(N^{-2}K^3) \geq 0$.
\end{proof}

\noi It follows from  (\ref{eq:N2K3}) that $N$ is determined by a choice of a square root of the canonical bundle $K$, thus revealing a new discrete invariant of a maximal semistable
$\Sp(4,\R)$-Higgs bundle with $w_1=0$ and $deg(L) = 2g-2$.  We introduce  subspaces of $\mathcal{M}^{max}$ as follows:

\begin{definition}\hfill
\begin{enumerate}
\item For $(w_1,w_2) \in H^1(X,\Z/2) \x H^2(X,\Z/2) \setminus (0,0)
\simeq ({(\Z/2)}^{2g}-\{0\})\times\Z/2$, define
\begin{equation}
\label{eq:Def_Mw1w2}
\mathcal{M}_{w_1,w_2} = \{(V,\beta,\gamma) \suchthat
w_1(V,\beta,\gamma) = w_1,\quad
w_2(V,\beta,\gamma) = w_2\} / \simeq,
\end{equation}
where the notation indicates isomorphism
classes of $\Sp(4,\R)$-Higgs bundles $(V,\beta,\gamma)$.
\item For $c \in H^2(X,\Z) \simeq \Z $ with $0 \leq c \leq 2g-2$, define
\begin{equation}
\label{eq:Def_M0c}
\mathcal{M}^{0}_{c} = \{(V,\beta,\gamma) \suchthat
w_1(V,\beta,\gamma) = 0, \quad
\deg(L) = c\} / \simeq,
\end{equation}
where $(W = L \oplus L^{-1},
q_W = \left(
\begin{smallmatrix}
0 & 1 \\
1 & 0
\end{smallmatrix}\right))$
is the Cayley  partner of $(V,\beta,\gamma)$.
\item For a square root $K^{1/2}$ of the canonical bundle, define the
following subspace of $\mathcal{M}^0_{2g-2}$
\begin{equation}
\mathcal{M}^{T}_{K^{1/2}} = \{(V= N \oplus N^{-1}K,\beta,\gamma) \suchthat
N = (K^{1/2})^3\} / \simeq.
\end{equation}
\end{enumerate}
\end{definition}
In particular, we can therefore write
\begin{equation}
\label{eq:def_MT}
\mathcal{M}^0_{2g-2} = \bigcup_{K^{1/2}}\mathcal{M}^{T}_{K^{1/2}},
\end{equation}
where $K^{1/2}$ ranges over the $2^{2g}$  square roots of
the canonical bundle.

\begin{remark}
\label{rem:Hitchin_component}
For the adjoint form of a split real reductive group $G$, Hitchin
showed in \cite{hitchin:1992} the existence of a distinguished
component of $\mathcal{M}(G)$, isomorphic to a vector space and
containing Teichm\"uller space. This component is known as the Hitchin
(or Teichm\"uller) component. In the case of $\Sp(4,\R)$, there are
$2^{2g}$ such components, which are exactly the components
$\mathcal{M}^T_{K^{1/2}}$\footnotemark\footnotetext{hence the
  superscript $T$ in the notation}. These components are all
projectively equivalent, in the sense that the restriction to each of
them of the projection to the moduli space for the adjoint group
$\SO_0(2,3) \simeq \mathrm{PSp}(4,\R)$ is an isomorphism onto the
unique Hitchin component in this moduli space (cf.\ \cite{GDP}).
\end{remark}

\begin{theorem}[\cite{gothen:2001}]
The subspaces $\mathcal{M}_{w_1,w_2}$, $\mathcal{M}^0_c$ with $0
\leq c < 2g-2$ and $\mathcal{M}^T_{K^{1/2}}$ are connected. Hence
the decomposition of $\mathcal{M}^{max}$ in its connected components
is
\begin{displaymath}
\mathcal{M}^{max}=(\bigcup_{w_1,w_2}\mathcal{M}_{w_1,w_2})\cup
(\bigcup_{0 \leq c < 2g-2}\mathcal{M}^0_c)\cup
(\bigcup_{K^{1/2}}\mathcal{M}^{T}_{K^{1/2}})
\end{displaymath}
and the total number of connected components is 
\begin{displaymath}
2(2^{2g}-1)+ (2g-2) + 2^{2g} = 3\cdot 2^{2g} +2g - 4\ .
\end{displaymath}
\end{theorem}

\noi The proof of the Theorem uses Hitchin's strategy
\cite{hitchin:1987,hitchin:1992} of considering the \textbf{Hitchin
function}, a positive proper function on the moduli space defined by
the $L^2$-norm of the Higgs field. Properness of the function
means that, in order to show that a given subspace $\mathcal{N}$ of
the moduli space is connected, it suffices to prove connectedness of
the non-empty subspace of local minima of the Hitchin function
restricted to $\mathcal{N}$.

\subsection{Description of the maximal Higgs bundles.}\label{sub:count-explicit}

The purpose of this section is to describe the Higgs bundles in each
connected component of $\mathcal{M}^{max}$.

\begin{proposition}\label{prop:stability}
Let $(V,\beta,\gamma)$ be an $\Sp(4,\R)$-Higgs bundle with $\deg(V) = 2g-2$.
\begin{enumerate}
\item Suppose that $V=N\oplus N^{-1}K$ and that with respect to this
decomposition, $\gamma=\left(
\begin{smallmatrix}
0 & 1 \\
1 & 0
\end{smallmatrix}\right)\in H^0(S^2V^*\otimes K)$,\ and $\beta=\left(
\begin{smallmatrix}
\beta_1 & \beta_3 \\
\beta_3 & \beta_2
\end{smallmatrix}\right)\in H^0(S^2V\otimes K)$.
\begin{enumerate}
\item If $g-1 < \deg(N)\le 3g-3$ then:
\begin{enumerate}
\item $(V,\beta,\gamma)$ is a stable $\Sp(4,\R)$-Higgs bundle if and
only if $\beta_2\ne 0$.
\item If $\beta_2 = 0$ then $(V,\beta,\gamma)$ is not semistable.
\end{enumerate}

\item If $\deg(N)=g-1$ then $(V,\beta,\gamma)$ is:
\begin{enumerate}
\item stable if and only if $\beta_2\ne 0$ and $\beta_1\ne 0$,
\item semistable if one of $\beta_2$ and $\beta_1$ is non-zero,
\item polystable if both $\beta_2=0$ and $\beta_1=0$.
\end{enumerate}

\end{enumerate}

\item If $V=W\otimes K^{1/2}$ where $W$ is as in (2) of Proposition \ref{prop:O(2)bundles} and  $\gamma=q_W\otimes 1_{K^{1/2}}$ then $(V,\beta,\gamma)$ is a stable $\Sp(4,\R)$-Higgs bundle.

\item If $V=(L_1\oplus L_2)\otimes K^{1/2}$ where $L_1$ and $L_2$ are line bundles satisfying $L_i^2=\cO$, $\gamma=\left(
\begin{smallmatrix}
q_1\otimes 1_{K^{1/2}} & 0 \\
0 & q_2\otimes 1_{K^{1/2}}
\end{smallmatrix}\right)$ where $q_i$ gives the isomorphism $L_i\simeq L^{-1}_i$ and $1_{K^{1/2}}$ denotes the identity map on $K^{1/2}$, and $\beta= \left(
\begin{smallmatrix}
\beta_1 & 0 \\
0 & \beta_2
\end{smallmatrix}\right)$, then
\begin{enumerate}
\item $(V,\beta,\gamma)$ is a polystable $\Sp(4,\R)$-Higgs bundle.
\item $(V,\beta,\gamma)$ is stable if and only if $L_1 \neq L_2$.
\end{enumerate}
\end{enumerate}
Moreover, if the $\Sp(4,\R)$-Higgs bundle $(V,\beta,\gamma)$ is stable
then it is simple, unless it is of the form described in Case~(3).
\end{proposition}

\begin{proof}
Part (1a) follows immediately from Proposition
\ref{prop: n=2stability-I} and the bounds on $\deg(N)$. Part (1b) follows
from Proposition \ref{prop:polystability}. Part (2) follows from the
fact that in this case $W$ is a stable $\OO(2)$-bundle. Part (3)
follows from Proposition~\ref{prop:polystability} and
Remark~\ref{rem:stable-not-simple}. 
\end{proof}

\begin{remark}\label{rem:smooth} 
Proposition \ref{prop:stability} (1a) says that for $0<c\le 2g-2$
all points in the components $\mathcal{M}^0_c$ are represented by
stable $\Sp(4,\R)$-Higgs bundles.  These Higgs bundles are,
moreover, simple and hence represent smooth points of the moduli
space (see \cite{garcia-prada-gothen-mundet:2009a}). It follows
that the components $\mathcal{M}^0_c$ are smooth for all $c$ in the
range $(0,2g-2]$.
\end{remark}

The following Proposition gives a description of the $\Sp(4,\R)$-Higgs
bundles in each component of $\mathcal{M}^{max}$. Everything in the Proposition follows
immediately from what we have said so far, except for the
identification of the minima of the Hitchin function (which, though
not essential, has been included for completeness; see
\cite{gothen:2001} for the proofs).

\begin{proposition}\label{prop:Vshape}
Let $[V,\beta,\gamma]$ denote an isomorphism class of  $\Sp(4,\R)$-Higgs bundles in  $\mathcal{M}^{max}$. Then
\begin{enumerate}
\item  $[V,\beta,\gamma]\in\mathcal{M}^{T}_{K^{1/2}}$ if and only if we can take $V=K^{3/2}\oplus K^{-1/2}$,
$\gamma=\left(
\begin{smallmatrix}
0 & 1 \\
1 & 0
\end{smallmatrix}\right)$, and $\beta= \left(
\begin{smallmatrix}
\beta_1 & \beta_3 \\
\beta_3 & 1_{K^{1/2}}
\end{smallmatrix}\right)$.  It represents a local minimum of the Hitchin function if and only
if $\beta_1 = 0$ and $\beta_3 = 0$.

\item $[V,\beta,\gamma]\in\mathcal{M}^0_c$ with $0<c<2g-2$ if and only if we can take $V=N\oplus
N^{-1}K$ where $N$ is a line bundle of degree $c+g-1$, $\gamma=\left(
\begin{smallmatrix}
0 & 1 \\
1 & 0
\end{smallmatrix}\right)$ and $\beta= \left(
\begin{smallmatrix}
\beta_1 & \beta_3 \\
\beta_3 & \beta_2
\end{smallmatrix}\right)$ with $\beta_2\ne 0$.  It represents a local minimum of the Hitchin function if and only
if $\beta_1 = 0$ and $\beta_3 = 0$.

\item $[V,\beta,\gamma]\in\mathcal{M}^{0}_{0}$ if and only if we can take $V=N\oplus N^{-1}K$ where $N$ is a line bundle of degree $g-1$ and
$\gamma=\left(
\begin{smallmatrix}
0 & 1 \\
1 & 0
\end{smallmatrix}\right)$.  It represents a local minimum of the Hitchin function if and only
if $\beta = 0$.

\item $[V,\beta,\gamma]\in\mathcal{M}_{w_1,w_2}$ if and only if we can take  either
\begin{enumerate}
\item $V=W\otimes K^{1/2}$ where $W$ is as in (2) of Proposition \ref{prop:O(2)bundles}, or
\item $V=L_1K^{1/2}\oplus L_2K^{1/2}$ where
\begin{enumerate}
\item $L_1$ and $L_2$ are line bundles satisfying $L_i^2=\cO$,
\item $w_1(L_1)+w_1(L_2)= w_1$, $w_1(L_1)w_1(L_2)= w_2$, and
\item $\gamma=\left(
\begin{smallmatrix}
q_1\otimes {\mathbf 1} & 0 \\
0 & q_2\otimes {\mathbf 1} 
\end{smallmatrix}\right)$ where ${\mathbf 1} $ denotes the identity map on $K^{1/2}$ and $q_i$ gives the isomorphism $L_i\simeq L^{-1}_i$.
\end{enumerate}
\end{enumerate} 
It represents a local minimum of the Hitchin function if and only
if $\beta = 0$.
\end{enumerate}
\end{proposition}

\begin{remark}
The $\Sp(4,\R)$-Higgs bundles of the type described in case (b) of
item (4) in Proposition~\ref{prop:Vshape} have $L_1 \neq L_2$ since
$w_1(L_1)+w_1(L_2)= w_1 \neq 0$. We point out that $\Sp(4,\R)$-Higgs
bundles of this form but with $L_1 = L_2$ ($\iff w_1(L_1)+w_1(L_2)=
0$) are isomorphic to those described in item (3) of the Proposition.
\end{remark}

\subsection{Description of maximal components.}\label{sub:components-explicit}

We can use the information in Section \ref{sub:count-explicit} to completely describe some components of $\mathcal{M}^{max}$.  Points in the moduli space correspond to isomorphism classes of Higgs bundles, while Proposition \ref{prop:Vshape} describes representatives of the isomorphism classes.  We thus need to understand when two representatives belong to the same isomorphism class. 

For $c$ in the range $0\le c\le 2g-2$, representatives of points in the components $\mathcal{M}^0_c$ are specified, according to Proposition \ref{prop:Vshape}, by tuples $(N,\beta_1,\beta_2,\beta_3)$ where $N$ is a line bundle of degree $c+g-1$, $\beta_1\in H^0(N^2K)$, $\beta_2\in H^0(N^{-2}K^3)$, and $\beta_3\in H^0(K^2)$. In the case $c=2g-2$ we require further that $N^2=K$ and that $\beta_2=1_{K^{1/2}}$.

\begin{proposition}\label{prop:equiv} Fix $c$ in the range $0<c\le 2g-2$. Tuples $(N,\beta_1,\beta_2,\beta_3)$ and $(N',\beta'_1,\beta'_2,\beta'_3)$ define the same isomorphism class  in $\mathcal{M}^0_c$ if and only if $N=N'$ and 
\begin{enumerate}
\item when $0<c< 2g-2$
$$(\beta'_1,\beta'_2,\beta'_3)=(t^2\beta_1,t^{-2}\beta_2,\beta_3)$$
\noi for some non-zero $t\in\C^*$, while
\item when $c=2g-2$
$$(\beta'_1,1_{K^{1/2}},\beta'_3)=(\beta_1,1_{K^{1/2}},\beta_3)\ .$$
\end{enumerate}
\end{proposition}

\begin{proof} Higgs bundles $(V,\beta,\gamma)$ and
  $(V',\beta',\gamma')$ are isomorphic if and only if there is a
  bundle isomorphism $g:V\longrightarrow V'$ such that
\begin{align}
\beta' &=g\otimes I_K \circ \beta\circ g^*\\
\gamma&=g^*\otimes I_K \circ \gamma'\circ g
\end{align}

\noi If the bundles are of the form $N\oplus N^{-1}K$ and $N'\oplus N'^{-1}K$, and if $\gamma=\gamma'=\begin{pmatrix}0&1\\1&0\end{pmatrix}$, then above constraints imply that there are two possibilities for $g$: either $g= \begin{pmatrix}0&g_2\\g_3&0\end{pmatrix}$ with $g_2 \in H^0(N'NK^{-1})$, $g_3  \in H^0((N'N)^{-1}K)$ and $g_2g_3=1$, or $g=\begin{pmatrix}g_1&0\\0&g_4\end{pmatrix}$ with $g_1 \in H^0(N'N^{-1})$, $g_4\in H^0(N'^{-1}N)$ and $g_1g_4=1$. The first possibility can occur only if $NN'=K$, i.e.\ if $c=0$.  Thus when $0<c\le 2g-2$, the only possibility is that $N^{-1}N'=\mathcal{O}$, i.e.\ $N=N'$, and that
\begin{equation}
g=\begin{pmatrix}t&0\\0&t^{-1}\end{pmatrix}
\end{equation}

\noi where $t$ is any non-zero complex number.  The result follows immediately from this.
\end{proof}

Let $\Jac^d$ be the Jacobian of degree $d$ line bundles on $X$ and let 
\begin{equation}
\mathcal{U}_d\longrightarrow \Jac^d(X)\times X
\end{equation}

\noi be the universal bundle. Denote the projections from $\Jac^d(X)\times X$ onto its factors $\Jac^d(X)$ and $X$ by $\pi_J$ and $\pi_X$ respectively. Define
\begin{align}
\mathcal{E}^{(1)}_d&={\pi_J}_*(\mathcal{U}_d^2\otimes\pi_X^*(K))\ ,\\
\mathcal{E}^{(2)}_d&={\pi_J}_*(\mathcal{U}_d^{-2}\otimes\pi_X^*(K^3))\ , \mathrm{and}\\
\mathcal{E}_{d}&=\mathcal{E}^{(1)}_d\oplus\mathcal{E}^{(2)}_d\oplus {\pi_J}_*(\pi_X^*(K^2))\ .
\end{align}

\noi Then $\mathcal{E}_d$ is a coherent sheaf over $\Jac^d$.  Moreover, for fixed $c$ in the range 
\begin{equation}0<c<g-1
\end{equation}
\noi  both $h^1(N^2K)$ and $h^1(N^{-2}K^3)$  vanish and thus, by the
Riemann-Roch theorem, $h^0(N^2K)$ and $h^0(N^{-2}K^3)$ are independent
of $N$.  It follows  that $\mathcal{E}_d$ is locally free with fiber
over the point  represented by the line bundle $N$ given by  
\begin{equation}
\mathcal{E}_{d,N}=H^0(N^2K)\oplus H^0(N^{-2}K^3)\oplus H^0(K^2)
\end{equation}

\begin{definition} \label{defn:C*} Define a $\C^*$-action on $\mathcal{E}_d$ by the fiberwise action
\begin{align}\label{eqtn:C*}
\C^*\times \mathcal{E}_{d,N} & \longrightarrow \mathcal{E}_{d,_N}\\
(t, (\beta_1,\beta_2,\beta_3))&\mapsto (t^2\beta_1,t^{-2}\beta_2,\beta_3)
\end{align}
\end{definition}

\begin{proposition}\label{prop:components}\hfil

(1) For each $c$ in the range $0< c< g-1$ the component $\mathcal{M}^0_c$ is the total space of the quotient $\hat{\mathcal{E}}_d/\C^*$ where 
\begin{itemize}
\item $d=c+g-1$,
\item $\hat{\mathcal{E}}_d$ denotes ${\mathcal{E}}_d$ minus the zero section of $\mathcal{E}^{(2)}_d$, and
\item  the $\C^*$ action is as in Definition (\ref{defn:C*}).
\end{itemize}

\noi  The fibers of  $\mathcal{M}^0_c$ as a fibration over $\Jac^{d}$ are given by  $\mathcal{O}_{\PP^{s}}(1)^{\oplus r}\times \C^{3g-3}$ where 
\begin{align}
r=&h^0(N^2K)=2c+3g-3\ \mathrm{and}\\
s= &h^0(N^{-2}K^3)-1=3g-4-2c\ .
\end{align}

(2) For each choice of a square root $K^{1/2}$ of the canonical bundle,
the component $\mathcal{M}^{T}_{K^{1/2}}$ is isomorphic to the vector
space $H^0(K^2) \oplus H^0(K^4)$.
\end{proposition}

\begin{remark} Part (2)  of this proposition is equivalent to Hitchin's
parametrization \cite{hitchin:1992} of his Teichm\"uller component
(cf.\ Remark~\ref{rem:Hitchin_component}).
\end{remark}

\begin{proof}
\noi Everything except the description of the fibers of $\mathcal{M}^0_c$ follows immediately from Propositions \ref{prop:Vshape} and \ref{prop:equiv}.  The description of the fibers follows from the claim that 
$$(\C^r\oplus (\C^*)^{s+1})/\C^*=\mathcal{O}_{\PP^{s}}(1)^{\oplus r}\ ,$$

\noi where the $\C^*$-action is given by $ t(\vec{z},\vec{w})= (t^2\vec{z}, t^{-2}\vec{w})$.  But the total space of $\mathcal{O}_{\PP^{s}}(1)^{\oplus r}$ can be identified with the variety

$$\mathcal{T}=\{(l,\vec{x_1},\dots,\vec{x_r})\ | l\ {\mathrm{defines\ a\ line\ in}}\ \C^{s+1}\ \mathrm{and}\ \vec{x_i}\in \C^{s+1} \ \mathrm{lies\ on}\ l\ \}$$

\noi The map
\begin{align}
(\C^r\oplus (\C^*)^{s+1})/\C^*&\longrightarrow \PP^s\times (\C^{s+1}\oplus\dots\oplus\C^{s+1})\\
[(z_1,\dots,z_r),\vec{w}]&\mapsto [\vec{w}], (z_1\vec{w},\dots, z_r\vec{w})
\end{align}

\noi is well defined with a well defined inverse on the subvariety
$\mathcal{T}$, and thus proves our claim.  The factor $\C^{3g-3}$
comes from $H^0(K^2)$.
\end{proof}

\begin{remark}  For $c$ in the range $[g-1,2g-2]$, there is a map $f: \mathcal{M}^0_c\longrightarrow \Jac^{c+g-1}$ but it is not surjective and the fiber dimension is not necessarily constant.  Nevertheless, by remark \ref{rem:smooth}, these components are smooth for all $c$ in the range $(0,2g-2]$.  
\end{remark}

\section{Subgroups for maximal representations}\label{sec:G*}

\subsection{Identification of possible subgroups.}\label{subs: subgroups}

The main result of this subsection,
Proposition~\ref{prop:possible-subgroups} identifies the possible
subgroups of $\Sp(4,\R)$ through which a maximal representation can
factor. The argument leading to this Proposition is due to Wienhard
\cite{wienhard:private}. The basis is the following result of Burger,
Iozzi and Wienhard
\cite{burger-iozzi-wienhard:2003,burger-iozzi-wienhard:2006}. 

\begin{theorem}\label{thm:biw}
Let $G$ be of hermitean type. Let $\rho\colon \pi_1(X) \to G$ be
maximal and let $\tilde{G} =
(\overline{\rho(\pi_1(X))}_{\R})^{\circ}$ (the identity component of
the real part of the Zariski closure). Then
\begin{enumerate}
\item $\tilde{G}$ is hermitean of tube type;
\item the embedding $\tilde{G} \into G$ is \emph{tight}.
\end{enumerate}
\end{theorem}

By classification of tube type domains (\cite{satake}) one has
the following.

\begin{lemma}\label{lem:small-tubes}
The only tube type domains of dimension less than or equal to $3$
and rank less than or equal to $2$ are $\D$, $\D \x \D$ and
$\Sp(4,\R) / \U(2)$.
\end{lemma}

\noi We identify three natural subgroups in $\Sp(4,\R)$ and then show that, as a result of Lemma \ref{lem:small-tubes}, these are essentially the only possibilities.  For two of them it is convenient to define $\Sp(4,\R)$ with respect to the symplectic form
\begin{equation}\label{defn:J12}
J_{12}=\begin{pmatrix}J & 0 \\0 & J\end{pmatrix}\ \mathrm{where}\  J=\begin{pmatrix}0 & 1 \\-1 & 0\end{pmatrix}\ .
\end{equation}

\begin{remark} The relation between $J_{12}$ and $J_{13}$ --- and hence between the resulting descriptions of $\Sp(4,\R)$ --- is described in Appendix \ref{subs:Sp4R}.
\end{remark}

\noindent  The subgroups come from the following three representations:

\smallskip
\begin{itemize}
\item  The irreducible $4$-dimensional representation of
$\SL(2,\R)$ in $\Sp(4,\R)$,
\begin{equation}
\rho_1\colon \SL(2,\R) \into \Sp(4,\R).
\end{equation}
\noi See Section \ref{sect: Gi} for a full description.
\item  The representation of $\SL(2,\R)\x\SL(2,\R)$ given with respect to $ J_{12}$ by
\begin{align}\label{eqn: product}
\rho_2\colon \SL(2,\R)\x\SL(2,\R)& \into \Sp(4,\R)\\
(A,B) &\mapsto
\begin{pmatrix}
A & 0 \\
0 & B
\end{pmatrix} \ .\nonumber
\end{align}
\item  The representation of $\SL(2,\R)$ given by,
\begin{equation}
\rho_3 = \rho_2 \circ \Delta \colon \SL(2,\R) \into \Sp(4,\R),
\end{equation}
where $\Delta$ is the diagonal embedding
$$\SL(2,\R) \hookrightarrow\SL(2,\R) \x \SL(2,\R)\ .$$
\end{itemize}

\begin{remark} Using the  Kronecker product\footnotemark\footnotetext{see Appendix \ref{subs:Sp4R}}, the diagonal embedding $\rho_3$ is given
by
\begin{equation}
\rho_3: A \mapsto \begin{cases} I\otimes A\ \mathrm{with\ respect\ to}\ J_{12}\\
A \otimes I \mathrm{with\ respect\ to}\ J_{13}\end{cases}
\end{equation}

\end{remark}

\begin{definition}
Let
\begin{align*}
\mathcal{D}_p &= \rho_p(\SL(2,\R)\x \SL(2,\R))/\rho_p(\SO(2)\x\SO(2));\\
\mathcal{D}_{\Delta} &= \rho_{\Delta}(\SL(2,\R))/\rho_{\Delta}(\SO(2)); \\
\mathcal{D}_i &= \rho_i(\SL(2,\R))/\rho_i(\SO(2)).
\end{align*}

\end{definition}

\noi With this notation, Lemma~\ref{lem:small-tubes}  together with the results of Wienhard et al. on tight embeddings (see \cite{burger-iozzi-wienhard:2007, wienhard:thesis}) implies the following.

\begin{proposition}\label{prop:subdomains}
Up to isometry of $\Sp(4,\R)/\U(2)$, the only proper tube type
domains tightly embedded in $\Sp(4,\R)/\U(2)$ are $\mathcal{D}_p
\simeq \D\x\D$, $\mathcal{D}_{\Delta} \simeq \D$ and
$\mathcal{D}_i\simeq \D$.
\end{proposition}

\begin{remark} Note that $\mathcal{D}_i\simeq \D$ is not holomorphically embedded,
while the other two are.  
\end{remark}

Proposition \ref{prop:subdomains}  is not quite sufficient for identifying the possible embedded subgroups since the subdomains do not uniquely determine the subgroups. Suppose that subgroups $G_1\subset G_2\subset \Sp(4,\R)$, with maximal compact subgroups $H_1\subset H_2$, give rise to the same subdomain,  i.e.\  are such that
$G_1/H_1=G_2/H_2$.  Then it is straightforward to see that
\begin{itemize}
\item $H_1$ is a normal subgroup of $H_2$, and
\item if the Cartan decompositions for the subgroups are $\lieg_i=\lieh_i+\liem_i$, then $\liem_1=\liem_2$.
\end{itemize}
It follows that  $G_1$ is a normal subgroup of $G_2$.  The next proposition is thus immediate.

\begin{proposition}
The following subgroups are the largest that give
rise to the embedded domains $\mathcal{D}_i, \mathcal{D}_p, $ and $\mathcal{D}_{\Delta}$ respectively:
\begin{align*}
G_i &= N_{\Sp(4,\R)}(\rho_1(\SL(2,\R))),\\
G_p &= N_{\Sp(4,\R)}(\rho_2(\SL(2,\R)\x\SL(2,\R))),\\
G_{\Delta} &= N_{\Sp(4,\R)}(\rho_3(\SL(2,\R))),
\end{align*}
\end{proposition}

\noi Hence Theorem~\ref{thm:biw} implies the following result. 

\begin{proposition}\label{prop:possible-subgroups}
Let $\rho\colon \pi_1(X) \to \Sp(4,\R) $ be
maximal and assume that $\rho$ factors through a proper reductive
subgroup $\tilde{G} \subset G$. Then, up to conjugation, $\tilde{G}$
is contained in one of the subgroups $G_i$, $G_{\Delta}$ and $G_{p}$.
\end{proposition}

\noi {\bf Note:} We will sometimes use $G_{*}$ to denote $G_i$, $G_p$ or $G_{\Delta}$.


\bigskip

\noi Explicit calculations show that:

\begin{proposition}\label{prop:GpandGdelta} We compute that

\noi (1)  $G_p$ is the group generated by $\SL(2,\R)\x\SL(2,\R)$ and $\begin{pmatrix}
0 & I \\
I & 0 \end{pmatrix}$.  That is, with respect to $J_{12}$, $G_p\subset\Sp(4,\R)$ is
\begin{equation*}
G_p=\{\begin{pmatrix}
X & Y \\
Z & T \end{pmatrix}\in\Sp(4,\R)\ |\ \mathrm{either}\ Y=Z=0\
\mathrm{or}\ X=T=0\ \}\ .
\end{equation*}

\noi (2)  $G_{\Delta}=\OO(2)\otimes \SL(2,\R)$ with respect to $J_{12}$.  That is, with respect to $J_{12}$, $G_{\Delta}\subset \Sp(4,\R)$ is
\begin{equation*}
G_{\Delta}=\{\begin{pmatrix}xA & yA \\zA & tA \end{pmatrix}\ |\ X=\begin{pmatrix}x & y
\\z & t
\end{pmatrix} \in \OO(2)\ \mathrm{and}\   A\in\SL(2,\R) \}\ .
\end{equation*}
\end{proposition}

\noi We defer the calculation of $G_i$ to Section \ref{sect: Gi} where the necessary details of the irreducible representation are given.   The result we obtain (see Proposition \ref{prop: Gi=SL2R}) is:

\begin{proposition}\label{prop:Gi}  $G_i=\SL(2,\R)$, i.e.
\begin{equation}
N_{\Sp(4,\R)}(\rho_1(\SL(2,\R)))= \rho_1(\SL(2,\R))\ .
\end{equation}

\end{proposition}

\section{Deformations of representations -- main results}
\label{sec:main-theorem}

\subsection{Invariants of representations}
\label{sec:invariants-reps}

Let $\rho\colon \pi_1(X) \to \Sp(4,\R)$ be a representation and let
$E_{\rho}$ be the associated flat $\Sp(4,\R)$-bundle. Then the Toledo
invariant $d(\rho)$ of $\rho$ is simply the first Chern class of the
(non-flat) $\U(2)$-bundle obtained by a reduction of the structure
group of $E_{\rho}$ to the maximal compact $\U(2) \subset \Sp(4,\R)$.
In terms of the $\Sp(4,\R)$-Higgs bundle $(V,\beta,\gamma)$ associated
to $\rho$ via the non-abelian Hodge theory correspondence, we have
$d(\rho) = \deg(V)$. A representation $\rho$ is said to be
\textbf{maximal} if $d(\rho) = 2g-2$ (cf.\ Proposition~\ref{prop:MW}).
Denote the subspace of maximal representations of
$\mathcal{R}(\Sp(4,\R))$ by $\mathcal{R}^{max}$. Then the non-abelian
Hodge theory correspondence (\ref{eq:non-ab-hodge}) gives a
homeomorphism
\begin{equation}
\label{eq:non-ab-hodge-max}
\mathcal{R}^{max} \simeq \mathcal{M}^{max}.
\end{equation}
We point out that, by the results of Burger, Iozzi and Wienhard
\cite{burger-iozzi-wienhard:2003,burger-iozzi-wienhard:2006}, any
maximal representation is reductive. Hence the space
$\mathcal{R}^{max}$ consists of all (isomorphism classes of) maximal
representations.

\begin{definition}
We denote by $\mathcal{R}_{w_1,w_2}$, $\mathcal{R}_{c}^{0}$ and
$\mathcal{R}^{T}_{K^{1/2}}$ the subspaces of $\mathcal{R}^{max}$
corresponding under (\ref{eq:non-ab-hodge-max}) to
the subspaces $\mathcal{M}_{w_1,w_2}$, $\mathcal{M}_{c}^{0}$ and
$\mathcal{M}^{T}_{K^{1/2}}$, respectively, of $\mathcal{M}^{max}$ (cf.\
(\ref{eq:Def_Mw1w2}), (\ref{eq:Def_M0c}) and (\ref{eq:def_MT})) .
\end{definition}

\begin{remark}
Though apparently of a holomorphic nature, the choice of a square
root $K^{1/2}$ of the canonical bundle of $X$ is in fact purely
topological: each such choice corresponds to the choice of a spin
structure on the oriented topological surface $S$ underlying $X$.
\end{remark}

\subsection{Main Theorem}
\label{sec:man-theorem}

With these preliminaries in place, we can state our main result.  The
proof is based on a careful analysis of $G_*$-Higgs bundles carried
out in Sections~\ref{sec:analysis-I}, \ref{sec:analysis-II} and
\ref{sect: Gi} below.

We shall say that a $\Sp(4,\R)$-Higgs bundle $(V,\beta,\gamma)$
\textbf{deforms} to a $\Sp(4,\R)$-Higgs bundle $(V',\beta',\gamma')$,
if they belong to the same connected component of the moduli space. In
other words, we mean continuous deformation through polystable
$\Sp(4,\R)$-Higgs bundles.  In the setting of representations, we use
the analogous notion of deformation.

\begin{theorem}\label{thm:main-result-higgs}
Let $X$ be a closed Riemann surface of genus $g \geq 2$ and let
$(V,\beta,\gamma)$ be a maximal polystable $\Sp(4,\R)$-Higgs
bundle. Then:

(1) $(V,\beta,\gamma)$ deforms 
to a polystable $G_{\Delta}$-Higgs bundle if and only if it belongs
to one of the subspaces $\mathcal{M}_{w_1,w_2}$ or
$\mathcal{M}_{0}^{0}$ of $\mathcal{M}^{max}$.

(2) $(V,\beta,\gamma)$ deforms 
to a polystable $G_{p}$-Higgs bundle if and only if it belongs
to one of the subspaces $\mathcal{M}_{w_1,w_2}$ or
$\mathcal{M}_{0}^{0}$ of $\mathcal{M}^{max}$.

(3)  $(V,\beta,\gamma)$ deforms 
to a polystable $G_{i}$-Higgs bundle if and only if it belongs
to one of the subspaces  $\mathcal{M}^{T}_{K^{1/2}}$.

(4) There is no proper reductive subgroup $G_* \subset \Sp(4,\R)$
such that $(V,\beta,\gamma)$ can be
deformed to a $G_*$-Higgs bundle if and only if $(V,\beta,\gamma)$
belongs to one of the components $\mathcal{M}_c^{0}$ with $0 < c <
2g-2$.
\end{theorem}

\noi The corresponding result for surface group representations is:

\begin{theorem}\label{thm:main-result}
Let $S$ be a closed oriented surface of genus $g \geq 2$ and let
$\rho\colon \pi_1(S) \to \Sp(4,\R)$ be a maximal
representation. Then:

(1) The representation $\rho$ deforms to a representation
which factors through the subgroup
$G_{\Delta} \subset \Sp(4,\R)$ if and only if it belongs to one of
the subspaces $\mathcal{R}_{w_1,w_2}$ or $\mathcal{R}_{0}^{0}$.

(2) The representation $\rho$ deforms to a representation
which factors through the subgroup
$G_{p} \subset \Sp(4,\R)$ if and only if it belongs to one of
the subspaces $\mathcal{R}_{w_1,w_2}$ or $\mathcal{R}_{0}^{0}$.

(3) The representation $\rho$ deforms to a representation
which factors through the subgroup
$G_{i} \subset \Sp(4,\R)$ if and only if it belongs to one of
the subspaces $\mathcal{R}^{T}_{K^{1/2}}$.

(4) There is no proper reductive subgroup $G_* \subset \Sp(4,\R)$
such that $\rho$ can be deformed to a
representation which factors through $G_*$ if and only if $\rho$
belongs to a component $\mathcal{R}_c^{0}$ for some $0 < c <
2g-2$.
\end{theorem}

\begin{proof}[Proof of Theorems~\ref{thm:main-result-higgs} and
\ref{thm:main-result}]
Statements (1)--(3) of Theorem~\ref{thm:main-result-higgs} follow
from the results for $G_*$-Higgs bundles given in Theorem~\ref{cor:
FinalTallyGdelta} for $G_* = G_{\Delta}$, Theorem~ \ref{th:GpTally} for
$G_* = G_{p}$ and Theorem~\ref{thm:GirrTally} for $G_* = G_{i}$.  

Statements (1)--(3) of Theorem~\ref{thm:main-result} now follow
immediately through the non-abelian Hodge theory correspondence
(\ref{eq:non-ab-hodge-max}).   Moreover, by Proposition~\ref{prop:possible-subgroups}, a maximal
representation which factors through a proper reductive subgroup
must in fact factor through one of the groups $G_{\Delta}$, $G_{p}$
and $G_{i}$. Hence statements (1)--(3) of Theorem~\ref{thm:main-result} imply
statement (4) of the same Theorem.

Finally, by the non-abelian Hodge theory correspondence
(\ref{eq:non-ab-hodge-max}), statement (4) of
Theorem~\ref{thm:main-result-higgs} follows from statement (4) of
Theorem~\ref{thm:main-result}.
\end{proof}

\begin{remark} Part (4) of this theorem says that for any
representation, say $\rho:\pi_1(X)\rightarrow\Sp(4,\R)$, represented
by a point in one of the components $\mathcal{R}^0_c$, the image
$\rho(\pi_1(X))$  in $\Sp(4,\R)$ is Zariski dense\footnote{We thank
Anna Wienhard for suggesting this formulation of the
result.}. Parts (1)--(3) of the theorem say that any other
representation can be deformed to one whose image is not Zariski
dense, and describe in which subgroups the image $\rho(\pi_1(X))$
may lie.
\end{remark}

\begin{remark}
Though (4) of Theorem~\ref{thm:main-result-higgs} is a result about
$\Sp(4,\R)$-Higgs bundles our proof depends on the correspondence
with representations, since it uses
Proposition~\ref{prop:possible-subgroups}. We expect,
though, that a pure Higgs bundle proof can be given by applying the
Cayley correspondence of \cite{GDP} (cf.\ Section~\ref{subs:cayley}).
\end{remark}

\section{Analysis of $G_{*}$-Higgs bundles I: $G_{\Delta}$-Higgs
bundles }
\label{sec:analysis-I}

In this section we identify the $\Sp(4,\R)$-Higgs bundles which
admit a reduction of structure group to $G_{\Delta}$, in the sense of
Definition~\ref{def:G-reduction}.

\subsection{The embedding of $G_{\Delta}$ in $\Sp(4,\R) $}

Proposition~\ref{prop:GpandGdelta} describes $G_{\Delta}$ as an
embedded subgroup of $\Sp(4,\R)$ (with respect to $J_{12}$).  As an
abstract group we can identify\footnotemark\footnotetext{The map
$(A,B)\mapsto B\otimes A$ defines a homomorphism from
$\SL(2,\R)\times \OO(2)$ to $\OO(2)\otimes \SL(2,\R)$
which is surjective and has kernel $\Z/2=\{\pm (I,I)\}$.} 
$G_{\Delta}$ as the group
\begin{equation}\label{eqn:GdeltaAbstract}
{G}_{\Delta}\simeq(\SL(2,\R)\times \OO(2))/(\Z/2)\ .
\end{equation}
This has a maximal
compact subgroup
\begin{equation}\label{eqn:Hdelta}
{H}_{\Delta}\simeq(\SO(2)\times \OO(2))/(\Z/2)\ 
\end{equation}

\noindent and a Cartan decomposition  of its Lie algebra 
\begin{equation}\label{eqn:GdeltaCartan}
\mathrm{Lie}({G}_{\Delta})\simeq(\lie{so}(2)\oplus\lie{o}(2))\oplus 
\liem(\SL(2,\R))
\end{equation}
where 
\begin{equation}\label{eqn:mSL2}
\liem(\SL(2,\R))=\{\begin{pmatrix}x & y\\y& -x\end{pmatrix}\in \lie{gl}(2,\R)\ \}\ .
\end{equation}

Since we prefer to use $J_{13}$ when describing $\Sp(4,\R)$-Higgs
bundles, we need to adjust the embedding given in Proposition
(\ref{prop:GpandGdelta}). Conjugation by the matrix $h$ given in
Appendix \ref{subs:Sp4R} shows that with respect to $J_{13}$ the
images of ${G}_{\Delta}$ and ${H}_{\Delta}$ are
\begin{align}\label{GdeltaJ13embed}
G_{\Delta}= & \SL(2,\R)\otimes \OO(2)
\\
=&\{\begin{pmatrix}aX & bX \\cX & dX \end{pmatrix} \suchthat  X^tX=I\ \text{ and}\  A=\begin{pmatrix}a & b
\\c & d
\end{pmatrix} \in \SL(2,\R)\ \}\ \\\nonumber
H_{\Delta}=  &\SO(2)\otimes \OO(2)\\
=& \{A\otimes X\in \SL(2,\R)\otimes \OO(2) \suchthat \ A^tA=I\ , \det(A)=1\ \}\nonumber
\end{align}

\begin{lemma} \hfill

\begin{enumerate}
\item The Lie algebra of $G_{\Delta}$ is invariant under the Cartan involution on $\Sp(4,\R)$.
\item $\Lie(G_{\Delta})\bigcap\frak{u}(2)=\Lie(H_{\Delta})$ where $\frak{u}(2)\subset\mathfrak{sp}(4,\R)$ is as in (\ref{eqn:u2})
\end{enumerate}
\end{lemma}

\begin{remark}
 \label{rem:GD-sp4r-compatible}
 It follows that the Cartan involution on $\Sp(4,\R)$ restricts to
 define a Cartan involution on $G_{\Delta}$.  In fact it is the
 Cartan involution\footnotemark corresponding to the decomposition
 (\ref{eqn:GdeltaCartan}) and we see that $G_{\Delta}$ is a reductive
 subgroup of $\Sp(4,\R)$ (see Section \ref{sec:reduction}). In
 particular, $H_{\Delta}$ lies in the $\U(2)$ subgroup embedded in
 $\Sp(4,\R)$ as in (\ref{eqn:U(2)}).
 \footnotetext{We cannot apply
 Remark~\ref{rem:ss-cartan-compatibility} directly to conclude this,
 because $G_{\Delta}$ is not semisimple. However, the explicit
 verification below of (\ref{CD:cartandata2}) justifies our claim.}
\end{remark}


The following computations are needed to identify the
$\Sp(4,\R)$-Higgs bundles whose structure group reduces to
$G_{\Delta}$.

\begin{proposition}

\noi (1) The complexification of ${G}_{\Delta}$ is
\begin{equation}\label{eqn:GdeltaCAbstract}
{G}^{\C}_{\Delta}\simeq(\SL(2,\C)\times \OO(2,\C))/(\Z/2)\ .
\end{equation}

\noi (2)  The complexification of ${H}_{\Delta}$ is  isomorphic to the complex conformal group, i.e.
\begin{equation}\label{eqn:HdeltaC}
{H}^{\C}_{\Delta}\simeq(\SO(2,\C)\times \OO(2,\C))/(\Z/2)\simeq\CO(2,\C)
\end{equation}
\noi where
\begin{equation}
\CO(2,\C)=\{A\in\GL(2,\C)\ |\ A^tA=\frac{\tr(A^tA)}{2}\ I\ \}
\end{equation}
\end{proposition}

\begin{proof} (1) Clear.  For (2) identify\footnotemark\footnotetext{ via
$$\lambda\mapsto\begin{pmatrix}\frac{\lambda+\lambda^{-1}}{2}&-\frac{\lambda-\lambda^{-1}}{2i}\\ \frac{\lambda-\lambda^{-1}}{2i}&\frac{\lambda+\lambda^{-1}}{2}\end{pmatrix}$$
}
$\SO(2,\C)$ with $\C^*$ and use the homomorphism
\begin{equation}\label{CO2projection}
\C^*\times \OO(2,\C)\longrightarrow\CO(2,\C)
\end{equation}

\noi defined by $(\lambda,A)\mapsto \lambda A$. This is surjective with kernel $\{\pm I\}$.
\end{proof}

\noi It follows from (\ref{eqn:GdeltaCAbstract}) and (\ref{eqn:HdeltaC}) that the complexification of the Cartan decomposition (\ref {eqn:GdeltaCartan}) is 
\begin{align}\label{eqn:GdeltaCartanC}
\mathrm{Lie}({G}^{\C}_{\Delta})=&\ \mathrm{Lie}({H}^{\C}_{\Delta})\oplus\liemc_{\Delta}\\
=&\ (\lie{so}(2,\C)\oplus\lie{o}(2,\C))\oplus \liem^{\C}(\SL(2,\R))\nonumber
\end{align}
\noi where 
\begin{equation}\label{eqn:mCSL2}
\liem^{\C}(\SL(2,\R))=\{\begin{pmatrix}x & y\\y& -x\end{pmatrix}\in \lie{gl}(2,\C)\ \}\ .
\end{equation}


\noi The proof of Proposition \ref{prop:GpandGdelta} `complexifies' to
show:

\begin{proposition}\label{prop:GdeltaCInSp4C}  The embedding of ${G}^{\C}_{\Delta}$ in $\Sp(4,\C)$ is given by
\begin{equation}\label{eqn:embedSL2CxOsC}
(A,X)\mapsto \begin{cases}
X\otimes A=\begin{pmatrix}xA & yA \\zA & tA \end{pmatrix}\ \text{with respect to}\ J_{12},\\
A\otimes X=\begin{pmatrix}aX & bX \\cX & dX \end{pmatrix}\ \text{with respect to}\ J_{13}
\end{cases}
\end{equation}

\noindent where $A=\begin{pmatrix}a & b
\\c & d\end{pmatrix}$ is in $\SL(2,\C)$ and $X=\begin{pmatrix}x & y
\\z & t
\end{pmatrix}$ is in $\OO(2,\C)$.
\end{proposition}

\noi These embeddings induce embeddings of $\mathrm{Lie}({G}^{\C}_{\Delta})$ in $\mathfrak{sp}(4,\C)$.  Let $\liemc_{\Delta}$
denote the image of $\liem^{\C}(\SL(2,\R))$ under the embedding with
respect to $J_{13}$.  It follows that we can identify
$\liemc_{\Delta}\subset\mathfrak{sp(4,\C)}$ as
\begin{equation}\label{eqn:mcdeltainsp}
\mathfrak{m}_{\Delta}^{\C}= \{
\begin{pmatrix}aI&bI\\bI&-aI\end{pmatrix} \suchthat a,b\in\C\ \}.
\end{equation}

\begin{remark} Comparison with the Cartan decomposition for
 $\Sp(4,\R)$ (see (\ref{eqn:CartanC}) and (\ref{eqn:gl2CInsp4C.2}))
 confirms that, as required (cf.\ (\ref{CD:cartandata}) and
 Remark~\ref{rem:GD-sp4r-compatible}), we get
\begin{equation}\label{CD:cartandata2}
\begin{CD}
\lie{sp}(4,\C) @. = @. \lie{gl}(2,\C) @. \oplus @. \liemc\\
@AAA @.  @AAA @. @AAA\\
\lieg^{\C}_{\Delta}@. = @. \lieh^{\C}_{\Delta}@. \oplus @. \liemc_{\Delta}\
\end{CD}
\end{equation}
\smallskip

\noi where $\lieg^{\C}_{\Delta}=\Lie(G^{\C}_{\Delta})$ and  $\lieh^{\C}_{\Delta}=\Lie(H^{\C}_{\Delta})$.  
\end{remark}

A change of basis via $T=\begin{pmatrix}1&i\\1&-i\end{pmatrix}$ transforms $\mathfrak{m}_{\Delta}^{\C}$ into
\begin{equation}\label{eqn:mcdeltainsp.2}
\mathfrak{m}_{\Delta}^{\C}= \{
\begin{pmatrix}0&\tilde{\beta}I\\\tilde{\gamma}I&0\end{pmatrix}\ |
\tilde{\beta}, \tilde{\gamma}\in\C\ \},
\end{equation}
where the descriptions in (\ref{eqn:mcdeltainsp}) and
(\ref{eqn:mcdeltainsp.2}) are related by
\begin{align}
\tilde{\beta}&=2(a+ib),\\
\tilde{\gamma}&=2(a-ib).
\end{align}

\subsection{The principal bundle}

\begin{lemma}\label{lemma:CO2} Let $V$ be a rank 2 vector bundle
associated to a principal $\CO(2,\C)$-bundle over $X$.  Fix a good
cover $\mathcal{U}=\{U_{\alpha}\}$ for $X$ and suppose that $V$ is
defined by transition functions $\{ g_{\alpha\beta}\}$ with respect
to $\mathcal{U}$. Pick $\{l_{\alpha\beta}\in\C^*\}$ and
$\{h_{\alpha\beta}\in \OO(2,\C)\}$ such that
\begin{equation}\label{eqn:gab}
g_{\alpha\beta}=l_{\alpha\beta}h_{\alpha\beta}.
\end{equation}
Then
\begin{enumerate}
\item the functions $\{ l^2_{\alpha\beta}\}$ define a line bundle, say $L$, and
\item $L^2=det^2(V)$
\end{enumerate}
\end{lemma}

\begin{proof}
Consider the cocycles $g_{\alpha\beta\gamma}$ defined by
\begin{align}
g_{\alpha\beta\gamma}=& g_{\alpha\beta}g_{\beta\gamma}g_{\gamma\alpha}\\
=& (l_{\alpha\beta}l_{\beta\gamma}l_{\gamma\alpha})(h_{\alpha\beta}h_{\beta\gamma}h_{\gamma\alpha})\ .\nonumber
\end{align}
Since $g_{\alpha\beta\gamma}=I$ and the $h_{\alpha\beta}$ are orthogonal , taking $g^t_{\alpha\beta\gamma}g_{\alpha\beta\gamma}$ yields
\begin{equation}
I=(l^2_{\alpha\beta}l^2_{\beta\gamma}l^2_{\gamma\alpha})I.
\end{equation}
This proves (1). Part (2) now follows directly from (\ref{eqn:gab}).
\end{proof}

\begin{remark} Using the description $\CO(2,\C)=(\OO(2,\C)\times\C^*)/(\Z/2)$, we can define a homomorphism
\begin{align}
\sigma:\CO(2,\C)&\longrightarrow \C^*\\
[A,\lambda]&\mapsto \lambda^2.
\end{align}
The bundle $L$ is the line bundle associated to $V$ by the
representation $\sigma$, i.e.\ if $E$ is the principal
$\CO(2,\C)$-bundle underlying $V$ then
\begin{equation}
L=E\times_{\sigma}\C.
\end{equation}
\end{remark}

The locally defined transition data $\{l_{\alpha\beta}\}$ or
$\{h_{\alpha\beta}\}$ do not in general define $\C^*$ or $\OO(2,\C)$
bundles.  However, if $V$ has even degree, then we get the following
decomposition.

\begin{lemma} Suppose $V$ and $L$ are as in Lemma \ref {lemma:CO2} and
  that $V$ has even degree. Then $\deg(L)$ is even and we can pick a
  line bundle $L_0$ such that
\begin{equation}
L_0^2=L.
\end{equation}
We can then decompose $V$ as
\begin{equation}
V=U\otimes L_0,
\end{equation}
where $U$ is an $\OO(2,\C)$ bundle.
\end{lemma}

\begin{proof} Using the same notation as in the proof of the previous
  lemma, let $L_0$ be defined by transition functions
  $\{n_{\alpha\beta}\}$. By construction we have
\begin{equation}
n_{\alpha\beta}^2=l_{\alpha\beta}^2\ .
\end{equation}
Moreover, the bundle $V\otimes L^{-1}_0$ is defined by transition functions
\begin{equation}
v_{\alpha\beta}=(\frac{l_{\alpha\beta}}{n_{\alpha\beta}})h_{\alpha\beta} \ .
\end{equation}
But then, since $h_{\alpha\beta}\in\OO(2,\C)$,
\begin{equation}
v^t_{\alpha\beta}v_{\alpha\beta}=(\frac{l^2_{\alpha\beta}}{n^2_{\alpha\beta}})h^t_{\alpha\beta} h_{\alpha\beta}=1.
\end{equation}
Thus $U=V\otimes L^{-1}_0$ is an $\OO(2,\C)$ bundle.
\end{proof}

\noi Conversely, we have the following. 

\begin{proposition}\label{prop:CO2}
If a rank $2$ vector bundle $V$ is of the form
\begin{equation}\label{eqn:V=UL}
V = U \otimes L_0,
\end{equation}
where $U$ is an $\OO(2,\C)$-bundle and $L_0$ is a line
bundle, then the structure group of $V$ reduces to $\CO(2,\C)$.
\end{proposition}

\begin{proof} The proof follows immediately from the projection
  (\ref{CO2projection}).
\end{proof}

\begin{remark}  It follows from (\ref{eqn:V=UL}) that the
line bundle $L_0$ must satisfy
\begin{equation}\label{eqn:L4=detV}
L_0^4= \det(V)^2\ .
\end{equation}
\end{remark}

\begin{corollary}\label{prop:maxtoledo3} Let $(V,\beta,\gamma)$ be a
  polystable $\Sp(4,\R)$-Higgs bundle with maximal Toledo invariant,
  i.e.\ with $\deg(V)=2g-2$. Then the structure group of $V$ (or,
  equivalently, of the underlying principal $\GL(2,\C)$-bundle)
  reduces to $\CO(2,\C)$, i.e.\ to $H^{\C}_{\Delta}$.
\end{corollary}

\begin{proof} By Proposition \ref{prop:maxtoledo2} we can write $V=U\otimes L_0$, as required by Proposition \ref{prop:CO2}.
\end{proof}

\subsection{The Higgs field}

By Lemma \ref{lemma:CO2} we can always give a `virtual' decomposition of a $\CO(2,\C)$ bundle $V$ as $V=U^v\otimes L_0^v$, where $U^v$ and $L_0^v$ are `virtual' bundles. This is an honest decomposition into actual bundles if $\deg(V)$ is even, and in all cases there is a line bundle $L$ such that $L=(L_0^v)^2$. 

\begin{proposition} Let $V=U^v\otimes L_0^v$ be the vector bundle in a $G_{\Delta}$-Higgs bundle.  The Higgs field is then a pair $(\tilde{\beta},\tilde{\gamma})$ where 
\begin{equation}\label{defn:GDeltaHiggsField}
\tilde{\beta}\in H^0((L_0^v)^2K))\ {\mathrm and}\  \tilde{\gamma}\in H^0((L^v_0)^{-2}K)\ .
\end{equation}
\end{proposition}

\begin{proof}  The Cartan decomposition of ${G}^{\C}_{\Delta} $ (see (\ref{eqn:GdeltaCartanC})) shows that the isotropy
representation of ${H}^{\C}_\Delta$ is given by
\begin{align*}
{H}^{\C}_\Delta = \C^* \x_{\pm 1} \OO(2,\C)
&\to \C^* \x \C^* \\
[\lambda,g] & \mapsto (\lambda^2,\lambda^{-2}).
\end{align*}

\noi Let  $E_{H^{\C}_\Delta}$ be the principal $\CO(2,\C)$ bundle underlying $V$. It  follows from the above observations that the bundle associated to $E_{H^{\C}_\Delta}$ by the isotropy representation, i.e.\ $E_{H^{\C}_\Delta}(\liem^{\C}_\Delta)=E_{H^{\C}_\Delta}\times_{\Ad}\liem^{\C}_\Delta$, is
\begin{equation}
E_{H^{\C}_\Delta}(\liem^{\C}_\Delta)=(L_0^v)^2\oplus (L_0^v)^{-2}\ .
\end{equation}
The result follows from this.
\end{proof}

\begin{proposition}\label{prop: GDeltaHiggsFields} Let $(V,\beta,
  \gamma)$ be an $\Sp(4,\R)$-Higgs bundle which admits a reduction of
  structure group to $G_{\Delta}$.  Then Higgs fields $\beta$ and
  $\gamma$ have to be of the form
\begin{align}
\beta&=\tilde{\beta}I,\\
\gamma&=\tilde{\gamma}I.
\end{align}
\end{proposition}

\begin{proof}
This is a direct consequence of (\ref{eqn:mcdeltainsp.2}).
\end{proof}

\noindent We can rephrase Proposition \ref{prop:
GDeltaHiggsFields} in a frame-independent way:

\begin{corollary}\label{cor:GDelta-higgs}
  Let $(V,\beta,\gamma)$ be a semistable $\Sp(4,\R)$-Higgs bundle for
  which the structure group reduces to $G_{\Delta}$. Suppose that $V$
  has a decomposition as $V =U\otimes L$ where $(U,q_U)$ is an
  orthogonal bundle and $L$ is a line bundle.  Then, using $S^2V =
  (S^2U) \otimes L^2$ and $S^2V^* = (S^2U^*) \otimes L^{-2}$, the
  components of the Higgs field are given by
\begin{displaymath}
\gamma=q_U \otimes \tilde{\gamma}\ ,\\
\beta= q_U^t \otimes \tilde{\beta}
\end{displaymath}
where
\begin{displaymath}
\tilde{\beta}\in H^0(L^2K),\quad \tilde{\gamma}\in H^0(L^{-2}K).
\end{displaymath}

\end{corollary}

\begin{remark} Notice that the section $\tilde{\gamma} \in
H^0(L^{-2}K)$ must be non-zero, since otherwise $\gamma =
\tilde{\gamma}q_U$ would be zero, contradicting semistability.  If
$\deg(V)=2g-2$ then $\deg(L)=g-1$ and $\deg(L^{-2}K) = 0$.  It follows
that in this case $L^2=K$, i.e.\ $L$ is a square root of $K$.
\end{remark}

\subsection{Identifying components with $G_{\Delta}$-Higgs bundles}

Having characterized $G_{\Delta}$-Higgs bundles, we now identify which
connected components of $\mathcal{M}^{max}$ contain the
$G_{\Delta}$-Higgs bundles.

\begin{theorem}Let $(V,\beta,\gamma)$ be a polystable
  $\Sp(4,\R)$-Higgs bundle with maximal (positive) Toledo invariant.
  If $(V,\beta,\gamma)$ represents a point in one of the components
  $\mathcal{M}^0_c$ with $0<c<2g-2$ or in one of the components
  $\mathcal{M}^{T}_{K^{1/2}}$, then the structure group of
  $(V,\beta,\gamma)$ does not reduce to $G_{\Delta}$.
\end{theorem}

\begin{proof} Let $(V,\beta,\gamma)$ be a polystable $\Sp(4,\R)$-Higgs
  bundle for which $\deg(V)=2g-2$.  Then $\gamma$ is an isomorphism
  and $V=W\otimes L_0$ where $W$ is an $\OO(2,\C)$-bundle and
  $L_0^2=K$ (see Section \ref{subs:cayley}).  Suppose that the
  structure group reduces to $G_{\Delta}$.  Then by Corollary
  \ref{cor:GDelta-higgs} and the remark following it, $V$ has a second
  decomposition $V=U\otimes L$ with $L^2=K$. Since the bundles in this
  decomposition are determined only up to a twist by a square root of
  the trivial line bundle, we can assume that $L=L_0$, and hence that
  $U=W$.  It follows, again by Corollary \ref{cor:GDelta-higgs}, that
  $ \beta=q^t\otimes\tilde{\beta}$ where $q$ is the quadratic form on
  $W$ and $\tilde{\beta}\in H^0(L^2K)$.

If $w_1=0$ then $V$ decomposes as
$$V=(L\oplus L^{-1})\otimes K^{1/2}=N\oplus N^{-1}K\ $$
and the quadratic form on $W=L\oplus L^{-1}$ is given by
$$ q=\left(
\begin{smallmatrix}
0 & 1 \\
1 & 0
\end{smallmatrix}
\right)\ .$$
It follows that
$$ \beta=\left(
\begin{smallmatrix}
0 & \tilde{\beta} \\
\tilde{\beta} & 0
\end{smallmatrix}
\right)$$
with respect to the decomposition $V=N\oplus N^{-1}K$.  A
comparison with the form of $\beta$ given in (1) and (2) of
Proposition \ref{prop:Vshape} shows that this is not possible if
$(V,\beta,\gamma)$ represents a point in $\mathcal{M}^0_c$ with
$0<c<2g-2$ or $\mathcal{M}^{T}_{K^{1/2}}$.
\end{proof}

Furthermore, by comparing our description of $G_{\Delta}$-Higgs
bundles with the descriptions of minima of the Hitchin function on
$\mathcal{M}^{max}$, and hence with the list of connected components
(see Section \ref{sub:count-explicit}), we get:

\begin{theorem}\label{cor: FinalTallyGdelta}
The following components of $\mathcal{M}^{max}$ contain
$G_{\Delta}$-Higgs bundles:
\begin{enumerate}
\item any component in which $w_1\ne 0$, i.e.
$$\mathcal{M}_{w_1,w_2}\  for\  any\ (w_1,w_2)\in
({(\Z/2)}^{2g}-\{0\})\times\Z/2\ ,$$
\item the component in which $w_1=0$ and $c_1=0$, i.e.\
$\mathcal{M}_{0}^{0}$.
\end{enumerate}
\end{theorem}

\begin{proof} We construct $\Sp(4,\R)$-Higgs bundles whose structure group reduces to $G_{\Delta}$ and show explicitly that they lie in the requisite components of $\mathcal{M}^{max}$.  Let $U$ be a stable $\OO(2,\C)$-bundle over $X$ and let $L$ be a square root of $K$.
Let $(w_1,w_2)$ be the first and second Stiefel-Whitney classes of
$U$ and let $q_U:U\longrightarrow U^*$ be the (symmetric)
isomorphism which defines the orthogonal structure on $U$.  Consider the
data $(V,\beta,\gamma)$, in which
\begin{itemize}
\item $V=U\otimes L$,
\item $\beta:V^*\longrightarrow VK$ is the zero map, and
\item $\gamma:V\longrightarrow V^*K$ is given by $q_U\otimes I_L$, where $I_L$ is the identity map on $L$,
\end{itemize}

By construction, the structure group of $V$ reduces to $\CO(2,\C)$
and the Higgs fields $\beta$ and $\gamma$ take values in
$\liemc_{\Delta}$. Thus $(V,\beta,\gamma)$ defines a
$G_{\Delta}$-Higgs bundle. It is polystable because the bundle $V$
is stable as an $\CO(2,\C)$ bundle.

If $(V,\beta,\gamma)$ is polystable as a $G_{\Delta}$-Higgs bundle
then it is polystable as an $\Sp(4,\R)$-Higgs bundle.  Since
$\deg(V)=2\deg(L)=2g-2$, it follows $(V,\beta,\gamma)$ lies in one of
the connected components of $\mathcal{M}^{max}$. As described in
Section \ref{sub:count}, the component containing $(V,\beta,\gamma)$
is labeled by invariants which classify the Cayley partner of
$(V,\beta,\gamma)$.  Since $L^2=K$ we may identify $U$ as the Cayley
partner. The invariants of $(V,\beta,\gamma)$ are thus $(w_1,w_2)$
if $w_1\ne 0$.  If $w_1=0$ then $U$ decomposes as
$$U=M\oplus M^{-1}\ $$
with $\deg(M)\ge 0$. The invariants of $U$ are then $(0,\deg(M))$. We
observe, finally, that $\deg(M)=0$ if $U$ is polystable.
\end{proof}

\section{Analysis of $G_{*}$-Higgs bundles II: $G_p$-Higgs bundles}
\label{sec:analysis-II}

\subsection{Generalities}

Recall the abstract description of $G_p$ as an extension
\begin{equation}
\label{eq:4}
\{1\} \to \SL(2,\R)\times \SL(2,\R) \to G_p \to \Z/2 \to \{0\},
\end{equation}
in fact, a semi-direct product
\begin{equation}\label{Gp_semidirect}
G_p = (\SL(2,\R)\times \SL(2,\R)) \rtimes \Z/2.
\end{equation}
Also,
\begin{proposition}
The maximal compact subgroups, $H_p\subset G_p$, and their 
complexifications $H^{\C}_p$ are conjugate to
\begin{align}
H_p&=(\SO(2)\times\SO(2))\rtimes\Z/2,\\
H^{\C}_p&=(\SO(2,\C)\times\SO(2,\C))\rtimes\Z/2\ .
\end{align}
\end{proposition}

With respect to $J_{13}$ the embedding (\ref{eqn: product}) becomes
\begin{equation}
(A,B)\longmapsto A\otimes \begin{pmatrix}1&0\\0&0\end{pmatrix}+B\otimes \begin{pmatrix}0&0\\0&1\end{pmatrix}
\end{equation}
showing that $\SO(2)\times\SO(2)$ (a maximal compact subgroup of
$\SL(2,\R)\times\SL(2,\R)$) embeds in the choice of maximal compact
subgroup of $\Sp(4,\R)$) (i.e.\ $\U(2)$) defined by (\ref{eqn:U(2)}).
After conjugation by $T=\begin{pmatrix}1&i\\1&-i\end{pmatrix}\otimes I$ this yields an embedding of $\SO(2,\C)\times\SO(2,\C)$ in $\SL(4,\C)$ given by 
\begin{equation}\label{eqn: SO2CxSO2Cembed}
\begin{pmatrix}u&-v\\ v&u\end{pmatrix},\begin{pmatrix}z&-w\\ w&z\end{pmatrix}\mapsto
\begin{pmatrix}u+iv&0&0&0\\ 0&z+iw&0&0\\ 0&0&u-iv&0\\ 0&0&0&z-iw\end{pmatrix}.
\end{equation}
Either way, since $\mathbb{G}_{p}$ is semisimple, it follows from
Remark~\ref{rem:ss-cartan-compatibility} that
\begin{proposition} The embedding defined in (1) of
Proposition~\ref{prop:GpandGdelta} makes $\mathbb{G}_{p}$ into a
reductive subgroup of $\Sp(4,\R)$.
\end{proposition}

Since the identification (\ref{Gp_semidirect}) induces an
isomorphism of Lie algebras
\begin{displaymath}
\liesl(2,\R)\times\liesl(2,\R) \to \Lie(G_p),
\end{displaymath}
we have the following result.

\begin{proposition}\label{prop: bundlesonly}
A $G_p$-Higgs bundle $(V,\beta,\gamma)$ admits a reduction
of structure group to $\SL(2,\R)\times\SL(2,\R)$ if and only if
the bundle $V$ admits a reduction of structure group from $H_p^{\C}$ to
$\SO(2,\C)\times\SO(2,\C)$.  \qed
\end{proposition}


\begin{proposition}\label{prop: Gpstructure}
If $(V,\beta,\gamma)$ is an $\Sp(4,\R)$-Higgs bundle for which the
structure group reduces to $\SL(2,\R)\times\SL(2,\R)$, then:
\begin{enumerate}
\item The bundle $V$ has the form
\begin{equation}
V=L_1\oplus L_2.
\end{equation}

\item The components of the Higgs field are diagonal with respect to
this decomposition, i.e.
\begin{equation}
\beta= \begin{pmatrix}\beta_1&0\\0&\beta_2\end{pmatrix}\ ,\
\quad \gamma= \begin{pmatrix}\gamma_1&0\\0&\gamma_2\end{pmatrix}
\end{equation}
with $\beta_i\in H^0(L_i^2K)$ and $\gamma_i\in
H^0(L_i^{-2}K)$.
\end{enumerate}
\end{proposition}

\begin{proof}
  For (1), apply (\ref{eqn: SO2CxSO2Cembed}) to the transition
  functions for the $\SO(2,\C)\times\SO(2,\C)$ bundle. As for (2), if
  the structure group of the Higgs bundle reduces to a subgroup $G_*$
  then the Higgs field takes vales in $\mathfrak{m}_*^{\C}\subset
  \mathfrak{m}^{\C}$ where $\mathfrak{m}_*^{\C}= \liegc_*/\liehc_*$,
  with the usual meanings for $\liegc_*, \liehc_*$, etc.  In our case
  , i.e.\ $G_*=\SL(2,\R)\times\SL(2,\R)$, expressed in global terms
  this means that $\beta$ must lie in
\begin{equation}
(L_1^2\oplus L_2^2)K\subset \Sym^2(L_1\oplus L_2)K
\end{equation}
and $\gamma$ must lie in
\begin{equation}
(L_1^{-2}\oplus L_2^{-2})K\subset \Sym^2(L^{-1}_1\oplus L^{-1}_2)K
\end{equation}
\end{proof}

\begin{remark} Proposition \ref{prop: Gpstructure} says
simply that if the structure group of $(V,\beta,\gamma)$ reduces to
$\SL(2,\R)\times\SL(2,\R)$, then $(V,\beta,\gamma)$ is a direct sum
of $\SL(2,\R)$-Higgs bundles, i.e.
\begin{equation}
(V,\beta,\gamma)=(L_1,\beta_1,\gamma_2)\oplus(L_2,\beta_2,\gamma_2).
\end{equation}
Of course for $(V,\beta,\gamma)$ to be {\it polystable} as
an $\Sp(4,\R)$-Higgs bundle, each $(L_i,\beta_i,\gamma_i)$ must be
(poly)stable as an $\SL(2,\R)$-Higgs bundle (cf.\
Remark~\ref{rem:SL2R-stability}). 

\end{remark}

\subsection{$G_p$-Higgs versus $\SL(2,\R)\times\SL(2,\R)$}\label{subs:method}

Let $(V,\beta,\gamma)$ be a $G_p$-Higgs bundle. The obstruction to
reducing the structure group to $\SL(2,\R)\times\SL(2,\R) \subseteq
G_p$ defines an invariant (depending, by Proposition \ref{prop:
bundlesonly}, only on $V$)
\begin{equation}\label{defn:xi}
\xi(V,\beta,\gamma) \in H^1(X,\Z/2).
\end{equation}
 Let $\{t_{\alpha\beta}\}$ be a {\v C}ech $\Z/2$-cocycle
representing the class $
\xi(V,\beta,\gamma)$ and  let
\begin{equation}\label{eqn:covermap}
p: X'\longrightarrow X
\end{equation}
 be an unramified double cover defined by
$\{t_{\alpha\beta}\}$. Note that if $\xi(V,\beta,\gamma)$ is non-zero then
\begin{equation}\label{eqn:g'}
g' = g(X') = 2g-1\ .
\end{equation}

\begin{proposition} Let $V'=p^*V$ be the pull-back of $V$ and let $\beta'=p^*\beta$ and $\gamma'=p^*\gamma$ be the pull-backs of the Higgs fields.
\begin{enumerate}
\item The bundle $V'$ admits a reduction of structure group to $\C^*\times\C^*$, i.e.\ we can write $V'$ as a sum of line bundles $L_1'\oplus L_2'$.
\item If $\iota:X'\longrightarrow X'$ is the involution covering the projection onto $X$ then $\iota^*(V')=V'$.
\item Both $\beta'$ and $\gamma'$ decompose, as $(\beta_1'\oplus\beta_2')$ and  $(\gamma_1'\oplus\gamma_2')$ respectively, with respect to the splitting $V'=L_1'\oplus L_2'$.
\item The pull-back of $G_p$-Higgs bundle $(V,\beta,\gamma)$ defines an $\SL(2,\R)\times\SL(2,\R)$-Higgs bundle, namely
\begin{equation}
p^*(V,\beta,\gamma)=(L_1',\beta_1',\gamma_1')\oplus
(L_2',\beta_2',\gamma_2')\ .
\end{equation}
\item If $(V,\beta,\gamma)$ is polystable and $\deg(V)=2g-2$, i.e.\ if $(V,\beta,\gamma)$ represents a point in $\mathcal{M}^{max}$, then in $(V',\beta',\gamma')$ we have

$$\deg(L_1)=\deg(L_2)=g'-1=2g-2\ .$$
\end{enumerate}
\end{proposition}

\begin{proof} Parts (1)--(4) follow by construction. It follows from
  (2) that $\deg(L_1)=\deg(L_2)=\frac{1}{2}\deg(V')$. Part (5) thus
  follows from (\ref{eqn:g'}) and
\begin{align*}
\deg(V')=\deg(\pi^*(V))&=\int_{X'} c_1(\pi^*(V))\\
&=\int_{\pi_*(X')} c_1(V)=2\int_{X}c_1(V)=2 \deg(V).
\end{align*}
\end{proof}

\subsection{Identifying components with $G_p$-Higgs bundles}

We now determine which components of $\mathcal{M}^{max}$ contain
Higgs bundles for which the structure group reduces to $G_p$
or to $\SL(2,\R)\times\SL(2,\R)$. In the next section we
consider components for which the invariant $w_1=0$, and in section
\ref{subs:w1ne0} we consider the case $w_1\ne 0$.

\subsubsection{The case $w_1= 0$.}

The invariant $w_1$ is the first Stiefel-Whitney class of the Cayley
partner of a maximal $\Sp(4,\R)$-Higgs bundle. Using the notation of
Section \ref {sub:count}, the connected components of
$\mathcal{M}^{max}$ in which $w_1=0$ are the components
$\mathcal{M}^0_c$ (with $0 \leq c<2g-2$) and the connected components of $\mathcal{M}^0_{2g-2}$ (i.e.\ the components $\mathcal{M}^T_{K^{1/2}}$).

\begin{proposition}\label{cor: w1=0}\hfil
\begin{enumerate}
\item For all $c$ in the range $0<c\le 2g-2$ the connected components of $\mathcal{M}^0_c$ do not contain
$\Sp(4,\R)$-Higgs bundles which admit a reduction of structure group
to $\SL(2,\R)\times\SL(2,\R)$.

\item The component $\mathcal{M}_{0}^{0}$ does contain $\Sp(4,\R)$-Higgs
bundles which admit a reduction of structure group to
$\SL(2,\R)\times\SL(2,\R)$ 
--- and hence to $G_p$. In fact the structure group can be reduced to the diagonally embedded $\SL(2,\R)\hookrightarrow \SL(2,\R)\times\SL(2,\R)$.
\end{enumerate}
\end{proposition}

\begin{proof}
Let $(V,\beta,\gamma)$ be a maximal $\Sp(4,\R)$-Higgs bundle.
Recall that $w_1 = 0$ means that
\begin{equation}
\label{eq:22}
V = N \oplus N^{-1}K,
\qquad \gamma =
\begin{pmatrix}
0 & 1 \\
1 & 0
\end{pmatrix},
\qquad \det(V) = K.
\end{equation}

\noindent Suppose furthermore that $(V,\beta,\gamma)$ admits a
reduction to $\SL(2,\R)\times\SL(2,\R)$. Then by Proposition
\ref{prop: Gpstructure}, together with the fact that it has maximal
Toledo invariant, this means that
\begin{equation}
\label{eq:33}
V = L_1 \oplus L_2,
\qquad L_1^2 = L_2^2 = K,
\qquad
\gamma =
\begin{pmatrix}
1 & 0 \\
0 & 1
\end{pmatrix}.
\end{equation}
For (\ref{eq:22}) and (\ref{eq:33}) to be compatible there must be
diagonal embeddings
\begin{displaymath}
L_\nu \into N \oplus N^{-1}K,\quad \nu=1,2.
\end{displaymath}
This is equivalent to
\begin{displaymath}
L_1 = L_2 = N=N^{-1}K
\end{displaymath}
\noindent and hence
\begin{displaymath}
K= L^2_1 = L^2_2 = N^2.
\end{displaymath}

\noindent In particular, $\deg(N)=g-1$, i.e.

\begin{equation}
c=\deg(N)-(g-1)=0\ .
\end{equation}

\noindent This proves (1).  To prove (2), pick any $L$ such that
$L^2=K$ and construct the $\SL(2,\R)$-Higgs bundle $(L,0,\gamma)$
with $\gamma=1_L$. Then the polystable Higgs bundle

$$ (L,0,\gamma)\oplus (L,0,\gamma)\ $$
proves part (2).

\end{proof}

Proposition \ref{cor: w1=0} leaves open the
possibility that there are $G_p$-Higgs bundles with  $w_1=0$ but in which the structure group does not reduce to
$\SL(2,\R)\times\SL(2,\R)$.  The next results rules out this
possibility.

\begin{proposition}
\label{prop:Gp-double-cover}
Let $(V,\beta,\gamma)$ be a maximal $G_p$-Higgs bundle which does not reduce
to an $\SL(2,\R)\times\SL(2,\R)$-Higgs bundle. Then, on the connected double cover
\begin{displaymath}
X' \xrightarrow{p} X
\end{displaymath}
defined by the class $\xi(V,\beta,\gamma)$, there exist line bundles
${L'}_1$ and ${L'}_2$ on $X'$ such that
\begin{displaymath}
p^*V = L'_1 \oplus L'_2,
\qquad {L'}_1^2 = {L'}_2^2 = K_{X'},
\qquad
p^*(\gamma) =
\begin{pmatrix}
1 & 0 \\
0 & 1
\end{pmatrix}.
\end{displaymath}
In other words, $p^*(V,\beta,\gamma)$ is a (maximal) Higgs bundle on
$X'$ with structure group $\SL(2,\R)\times\SL(2,\R)$.
\end{proposition}
\begin{proof}
Clear.
\end{proof}

\begin{proposition}
Let $(V,\beta,\gamma)$ be a maximal $G_p$-Higgs bundle for which the structure group does not reduce
to $\SL(2,\R)\times\SL(2,\R)$. Assume that $w_1(V,\beta,\gamma)= 0$, in other words, that $(V,\beta,\gamma)$ is of the form (\ref{eq:22}). Then $\deg(N) = g-1$.
\end{proposition}

\begin{proof}
Combining Propositions~\ref{cor: w1=0} and
\ref{prop:Gp-double-cover} we get that
\begin{displaymath}
(p^*N)^2 = K_{X'}.
\end{displaymath}
Recall, moreover, that $g(X')=2g(X)-1$ and that $\deg(p^*N)=2\deg(N)$.
The result now follows.
\end{proof}

\begin{corollary}\label{cor: NoGp}
None of the components $\mathcal{M}^0_c$ with $c > 0$ contains $\Sp(4,\R)$-Higgs
bundles  which admit a reduction of structure group to $G_p$.
\end{corollary}

\subsubsection{The case $w_1\ne 0$}\label{subs:w1ne0}

In this section we prove the following.

\begin{proposition}\label{w1ne0}
For all $(w_1,w_2)\in (H^1(X,\Z/2)-\{0\})\times H^2(X,\Z/2)$ the
component $\mathcal{M}_{w_1,w_2}$ contains $\Sp(4,\R)$-Higgs
bundles which admit a reduction of structure group to
$\SL(2,\R)\times\SL(2,\R)\subset G_p$.
\end{proposition}

\begin{proof}

Let $(V,\beta,\gamma)$ be a $\Sp(4,\R)$-Higgs bundle of the form
\begin{displaymath}
V = L_1 \oplus L_2,
\qquad L_1^2 = L_2^2 = K,
\qquad \beta=0\ ,\qquad
\gamma =
\begin{pmatrix}
1 & 0 \\
0 & 1
\end{pmatrix}.
\end{displaymath}

If we fix a square-root of $K$, i.e.\ if we pick $L_0$ such that
$L_0^2=K$, and define the Cayley partner $W=V^*\otimes L_0$, then
we get
\begin{equation}
W=M_1\oplus M_2
\end{equation}
 with $M_i^2=\cO$. Moreover, $\gamma$ defines isomorphisms
\begin{equation}
\tilde{\gamma}_i:M_i\longrightarrow M_i^*\ ,
\end{equation}
 that is, $M_1$ and $M_2$ are $\OO(1,\C)$ bundles.  As
such, they are determined by their first Stiefel--Whitney classes
\begin{displaymath}
w_1(M_1),w_1(M_2) \in H^1(X,\Z/2).
\end{displaymath}
To determine the invariants of $W$, we need to calculate the total
Stiefel--Whitney class
\begin{align}
w(M_1\oplus M_2) &= 1 + w_1(M_1\oplus M_2) + w_2(M_1\oplus M_2)\\
&= 1 + w_1(M_1) + w_1(M_2) + w_1(M_1)w_1(M_2).
\end{align}
In other words, we need to analyze the map
\begin{displaymath}
\begin{aligned}
H^1(X,\Z/2) \x H^1(X,\Z/2) &\to H^1(X,\Z/2) \x H^2(X,\Z/2) \\
(w_1,w_1') & \mapsto (w_1+w_1', w_1w_1').
\end{aligned}
\end{displaymath}
Using standard coordinates on $H^1(X,\Z/2) \cong {(\Z/2)}^{2g}$ we write
an element in this space as $(\ua,\ub) =
\bigl((a_1,\dots,a_g),(b_1,\dots,b_g)\bigr)$. The map is then given
as follows:
\begin{equation}
\label{eq:6}
\begin{aligned}
{(\Z/2)}^{2g}\x {(\Z/2)}^{2g} &\to {(\Z/2)}^{2g}\x \Z/2, \\
\bigl((\ua,\ub),(\ua',\ub')\bigr) &\mapsto
((\ua+\ua',\ub+\ub'),\sum_{i=1}^{g}(a_ib_i'+a_i'b_i)).
\end{aligned}
\end{equation}
One easily sees that $a_i+a_i'=0$ and $b_i+b_i'=0$ imply that
$a_ib_i'+a_i'b_i = 0$.
Moreover, one has that
\begin{displaymath}
\bigl((\ua,\ub),(0,0)\bigr) \mapsto \bigl((\ua,\ub),0\bigr).
\end{displaymath}
Hence it only remains to show that any element of the form
$\bigl((\tilde\ua,\tilde\ub),1\bigr)$ with $(\tilde{a}_j,\tilde{b}_j)
\neq (0,0)$ for some $j$ is in the image of the map.  It is a
simple exercise to show that there exists $((a,b),(a',b')) \in
{(\Z/2)}^{2}\x {(\Z/2)}^{2}$ such that $ab'+a'b = 1$ and $(a+a',
b+b')=(\tilde{a}_j,\tilde{b}_j)$. Now let $(\ua,\ub)$ be the element
obtained from $(\tilde\ua,\tilde\ub)$ by substituting $a$ for
$\tilde{a}_j$ and $b$ for $\tilde{b}_j$. Moreover, define $(\ua',\ub')$ by
letting the $j$th entries of $\ua'$ and $\ub'$ be equal to $a'$ and
$b'$, respectively, and setting the remaining entries equal to
zero. Then, clearly,
\begin{displaymath}
\bigl((\ua,\ub),(\ua',\ub')\bigr) \mapsto \bigl((\tilde\ua,\tilde\ub),1\bigr),
\end{displaymath}
and this concludes the proof.
\end{proof}

\subsection{The final tally}

Combining Corollary \ref{cor: NoGp} and Proposition \ref{w1ne0} we
get, finally, that

\begin{theorem}\label{th:GpTally} The following components of $\mathcal{M}^{max}$ contain $\Sp(4,\R)$-Higgs bundles which admit a reduction of structure group to the subgroup $\SL(2,\R)\times\SL(2,\R)\subset G_p$:
\begin{itemize}
\item  $\mathcal{M}_{w_1,w_2}$, for all $(w_1,w_2)\in H^1(X,\Z/2)-\{0\}\times H^2(X,\Z/2)$,
\item $\mathcal{M}_{0}^{0}$
\end{itemize}

\noindent In the remaining components, i.e.\ in $\mathcal{M}^0_c$ for
$0<c<2g-2$ and in $\mathcal{M}^T_{K^{1/2}}$ for all choices of
$K^{1/2}$, none of the Higgs bundles admit a reduction of structure
group to $G_p$. \qed
\end{theorem}

\section{Analysis of $G_{*}$-Higgs bundles III: $G_i$-Higgs bundles}\label{sect: Gi}

\subsection{The irreducible representation}
\label{subs:irep}

The irreducible representation of $\SL(2,\R)$ in $\R^4$ comes from its representation on $S^3\R^2$, the third symmetric tensor power of $\R^2$.  If we
identify $S^3\R^2$ with the space of degree three homogeneous
polynomials in two variables, then the representation is defined by
\begin{equation}
\rho_1(\begin{pmatrix} a&b\\ c&d\end{pmatrix})(P)(x,y)=P(ax+cy,bx+dy),
\end{equation}
where $A=\begin{pmatrix} a&b\\ c&d\end{pmatrix}$ is in $\SL(2,\R)$ and
$P$ is a degree three homogeneous polynomial in $(x,y)$. We get a
matrix representation (denoted by $\rho_1$) if we fix a basis for $S^3\R^2$.  Taking
$$\{x^3, 3x^2y, y^3, 3xy^2\}$$
as our basis for $S^3\R^2$ (thought of as the space of degree three homogeneous polynomials in two variables) we get 
\begin{displaymath}\label{eq:irrepmatrix}
\rho_1(\begin{pmatrix}
a & b \\
c & d
\end{pmatrix}) =
\begin{pmatrix}
a^3 & 3a^2b & b^3 & 3ab^2  \\
a^2c & a^2d + 2abc & b^2d & b^2c + 2abd  \\
c^3 & 3c^2d & d^3 & 3cd^2\\
ac^2 & bc^2 + 2acd & bd^2 & ad^2 + 2bcd 
\end{pmatrix}.
\end{displaymath}

The standard symplectic form $\omega = dx_1 \wedge dx_2$ on $\R^2$ induces a
bilinear form on all tensor powers $(\R^2)^{\otimes n}$, as follows:
\begin{displaymath}
\Omega((v_1,\dots,v_n),(w_1,\dots,w_n)) =
\omega(v_1,w_1)\cdot\dots\cdot \omega(v_n,w_n),
\end{displaymath}
and therefore there is also an induced bilinear form on the symmetric
powers of $\R^2$, viewed as subspaces $S^n \R^2 \subset
(\R^2)^{\otimes n}$. This form is symmetric when $n$ is even and
antisymmetric when $n$ is odd so, in particular, gives us a symplectic
form $\Omega$ on $S^3\R^2$. Non-degeneracy follows from the fact that
the kernel of the form is an $\SL(2,\R)$-submodule of an irreducible
representation (and can of course also be seen from the explicit
calculation below).

Take the standard basis $\{e_1,e_2\}$ of $\R^2$ and the basis
$$
\{e_{ijk} = e_i\otimes e_j \otimes e_k \suchthat i,j,k = 1,2\}
$$
of $(\R^2)^{\otimes 3}$. Then the
basis $\{E_1,E_2,E_3,E_4\}$ for $S^3\R^2$, where
\begin{align*}
E_1 &= e_{111},\\
E_2 &= e_{112} + e_{121} + e_{211},\\
E_3 &= e_{222},\\
E_4 &= e_{122} + e_{212} + e_{221}
\end{align*}
corresponds to the basis $\{x^3, 3x^2y, y^3, 3xy^2\}$ for $S^3\R^2$
thought of as the space of degree three homogeneous polynomials of
degree in two variables.  Calculating the matrix $J_0$ of the
symplectic form $\Omega$ on $S^3\R^2$ with respect to this basis one
obtains:
\begin{displaymath}
J_0 =
\begin{pmatrix}
0 & 0 & 1& 0 \\
0 & 0 & 0 & -3 \\
-1 & 0 & 0 & 0\\
0 & 3 & 0 & 0 
\end{pmatrix}.
\end{displaymath}
If  $ad -bc = 1$, i.e.\ if $A=\begin{pmatrix} a&b\\ c&d\end{pmatrix}$ is symplectic, then $\rho_1(A)$ is a
symplectic transformation of $(S^3,\R^2,\Omega)$ , i.e.
\begin{equation}\label{eq:symp-J_0}
\rho_1(A)^tJ_0\rho_1(A)=J_0\ .
\end{equation}
\noi Notice that with
\begin{displaymath}
h =
\begin{pmatrix}
1 & 0 & 0 & 0 \\
0 & 0 & 0 & \frac{1}{\sqrt{3}}  \\
0 & 0 & 1 & 0\\
0 & \frac{1}{\sqrt{3}} & 0 & 0
\end{pmatrix} = h^t
\end{displaymath}
we get
\begin{displaymath}
h^t J_0 h =J_{13}\ .
\end{displaymath}
Thus using  $J_{13}$ to define $\Sp(4,\R)$, the irreducible representation is given by
\begin{equation}\label{irredJ13}
\rho_{13}(A)= h^{-1}\rho_1(A)h =
\begin{pmatrix}
a^3 & \sqrt{3}ab^2 & b^3 & \sqrt{3}a^2b  \\
\sqrt{3}ac^2 & ad^2 + 2bcd & \sqrt{3}bd^2 & bc^2 + 2acd \\
c^3 & \sqrt{3}cd^2 &d^3 & \sqrt{3}c^2d\\
\sqrt{3}a^2c & b^2c + 2abd & \sqrt{3}b^2d & a^2d + 2abc
\end{pmatrix}
\end{equation}
If $A\in \SO(2)$, i.e.\ if $d=a,\ b=-c$ and $a^2+c^2=1$, then $\rho_{13}\begin{pmatrix}
a & -c \\
c & a
\end{pmatrix}$ lies in the copy of $\U(2)$ embedded in
$\Sp(4,\R)$ as in (\ref{eqn:U(2)}).  
Moreover, with the Cartan involution as in (\ref{eqn:theta}) the image of the induced embedding 
\begin{equation}
\rho_{13*}:\mathfrak{sl}(2,\R)\longrightarrow\mathfrak{sp}(4,\R)\ .
\end{equation}
is $\theta$-invariant, so
Remark~\ref{rem:ss-cartan-compatibility} gives us the following.

\begin{proposition} With $G_i = \rho_{13}(\SL(2,\R))$  defined as above and with the choices for $\Sp(4,\R)$ as in Section \ref{subs:Sp4cartan}, the subgroup $G_i $ is a reductive subgroup of $\Sp(4,\R)$.
\end{proposition}

\begin{remark}\label{rem:embedSL2C}
This embedding extends to an embedding of $\SL(2,\C)$ in
$\Sp(4,\C)\subset\SL(4,\C)$.  The restriction to $\SO(2,\C)$ takes
values in the copy of $\GL(2,\C)$ embedded in $\SL(4,\C)$ via

\begin{equation}\label{eqn:GL2embed}
Z\mapsto\begin{pmatrix}\frac{Z+{Z^t}^{-1}}{2}&\frac{Z-{Z^t}^{-1}}{2i}\\
\frac{{Z^t}^{-1}-Z}{2i}&\frac{Z+{Z^t}^{-1}}{2}
\end{pmatrix}\ .
\end{equation}
\end{remark}

If we conjugate by $T=\begin{pmatrix}
I & iI \\
I &- iI
\end{pmatrix}$, that is if we make a complex change of frame from $\R^4\otimes\C$ to $\C^2\oplus (\C^2)^*$ , the embedding of $\SO(2)$ becomes (with $A=\begin{pmatrix}
a & -c \\
c & a
\end{pmatrix}$) 
\begin{equation*}
T\circ\rho_{13}(A)\circ T^{-1}=\begin{pmatrix}\Lambda&0_2\\0_2&(\Lambda^t)^{-1}\end{pmatrix}
\end{equation*}

\noi where $0_2$ denotes the $2\times 2$ zero matrix and
\begin{equation*}
\Lambda=\begin{pmatrix}
a^3+ic^3 & \sqrt{3}ac(ia+c) \\
\sqrt{3}ac(ia+c)  & a^3-2ac^2 +i (c^3-2a^2c) \end{pmatrix}\ .
\end{equation*}

\noindent A further conjugation  by
\begin{equation}\label{eqn: Htilde}
\tilde{H}= \begin{pmatrix}
0 & 0 & \frac{\sqrt{3}-1}{8}u   & \frac{ \sqrt{3}-3}{8}u \\
0 & 0 &-\frac{ \sqrt{3}+3}{8} v& -\frac{\sqrt{3}+1}{8}v\\
\frac{\sqrt{3}+1}{u} & -\frac{(\sqrt{3}+3)}{u}    &0 &0\\
\frac{\sqrt{3}-3}{v}  & -\frac{(\sqrt{3}-1)}{v} & 0 & 0\\
\end{pmatrix}\ ,
\end{equation}

\noindent where $u= -4 \sqrt{6 + 3 \sqrt{3}}$ and
$v=2/\sqrt{2+\sqrt{3}}$, yields
$$ 
\tilde{H}\circ T\circ \rho_{13}(A)\circ (\tilde{H}\circ T)^{-1}= \begin{pmatrix}
\lambda^{3} & 0 & 0 &0 \\
0 & \lambda^{-1} & 0 &0\\
0 &0 & \lambda^{-3} & 0\\
0 &0 & 0 & \lambda^{1}\\
\end{pmatrix}\ ,\ \lambda=a+ic\ .
$$

\begin{remark}
Direct computation shows that with $\Sp(4,\C)$ defined by $J_{13}$, conjugation by $T$ or $\tilde{H}$ preserves $\Sp(4,\C)\subset\SL(4,\C)$. 
\end{remark}


\begin{definition} Let $\varphi:\SL(2,\C)\longrightarrow\Sp(4,\C)$ be
the composite
\begin{equation}\label{eqn:varphi}
\varphi(A)=(\tilde{H}\circ T)\circ\rho_{13}(A)\circ(\tilde{H}\circ T)^{-1}\ .
\end{equation}
\end{definition}

\noindent
We then have a commutative diagram 
\begin{equation}
\label{eq:restrict-phi}
\begin{CD}
\SL(2,\C) @>{\varphi}>> \Sp(4,\C) \\
@AAA  @AAA \\
\GL(1,\C) @>{\varphi_{|\GL(1,\C)}}>> \GL(2,\C)
\end{CD}
\end{equation}
where the vertical arrow on the left is given by the identification
\begin{equation}\label{eqn:GL1C}
\GL(1,\C) \simeq \Big\{\begin{pmatrix}\lambda&0\\0&\lambda^{-1}\end{pmatrix}\ |\ \lambda\in\C^*\Big\}\end{equation}

\noi and the one on the
right is given by (\ref{eqn:GL2embed}).

\subsection{The embedding of Higgs bundles}\label{subs: Lie}

We can compute the infinitesimal version of the embedding
(\ref{irredJ13}) to find the embedding of
$\lie{sl}(2,\R)\subset\lie{sp}(4,\R)$ (using $J=J_{13}$). With
\begin{displaymath}
e=\begin{pmatrix}
0 & 1\\
0 & 0 
\end{pmatrix}\ ,\
f=\begin{pmatrix}
0 & 0\\
1 & 0
\end{pmatrix}\ ,\
h_0=\begin{pmatrix}
1 & 0\\
0 & -1
\end{pmatrix}\
\end{displaymath}

\noindent and with $\tilde{H}$ and $T$ as above, we compute
\begin{lemma} 
\begin{align*}
(\tilde{H}T)\rho_{13*}(e-f)(\tilde{H}T))^{-1}&=i \begin{pmatrix}-3 & 0 & 0& 0\\0  & 1 & 0 & 0\\
0  & 0 & 3& 0\\0  & 0 & 0 & -1\end{pmatrix}\\
(\tilde{H}T)\rho_{13*}(e+f)(\tilde{H}T))^{-1}&=i \begin{pmatrix}
0 & 0&0 &3 \\
0 & 0&3 & -1 \\
0 &-1& 0 & 0 \\
-1 & 4& 0 & 0 
\end{pmatrix}\\
(\tilde{H}T)\rho_{13*}(h_0)(\tilde{H}T)^{-1}&= \begin{pmatrix}
0 & 0&0 &3 \\
0 & 0&3 & 1 \\
0 &1& 0 & 0 \\
1 & 4& 0 & 0 
\end{pmatrix}
\end{align*}

\end{lemma}

\begin{proof} Calculation (Mathematica).
\end{proof}

\noindent  It follows that the restriction of $\varphi$ to  $\liemc(\SL(2,\C))$,where
\begin{equation*}
\liemc(\SL(2,\C))= \{\begin{pmatrix} x&y\\y&-x\end{pmatrix}\ |\  x,y\in\C \}\ \ ,
\end{equation*}

\noindent gives
\begin{equation*}
(\tilde{H}T)\rho_{13*}(\begin{pmatrix} x&y\\y&-x\end{pmatrix})(\tilde{H}T))^{-1}=
\begin{pmatrix}
0 & 0 &  0 &3 \beta\\
0 & 0 &  3\beta & \gamma \\
0 &\gamma& 0 & 0 \\
\gamma & 4\beta &  0 & 0 
\end{pmatrix} \ \mathrm{with}\ \begin{cases}\beta= x+iy\\
\gamma=x-iy\end{cases}\ .
\end{equation*}

\noi We can make a further transformation  so that the bottom left corner is a multiple of  $\begin{pmatrix}0&1\\1&0\end{pmatrix}$.  

\begin{lemma}
Let
\begin{equation}
S= \begin{pmatrix}
1 &  2(\frac{\beta}{\gamma} )& 0 & 0\\
0 & 1 & 0 & 0 \\
0 & 0 & 1 & 0\\
0 & 0 & -2(\frac{\beta}{\gamma} ) & 1 \\
\end{pmatrix} 
\end{equation}

\noindent Then
\begin{equation*}
(S\tilde{H}T)\rho_{13*}(\begin{pmatrix} x&y\\y&-x\end{pmatrix})(S\tilde{H}T))^{-1}=
\gamma\begin{pmatrix}
0 & 0 &  16(\frac{\beta}{\gamma})^2 &5 (\frac{\beta}{\gamma})\\
0 & 0 &  5(\frac{\beta}{\gamma}) & 1 \\
0 &1& 0 & 0 \\
1& 0 &  0 & 0 \\
\end{pmatrix} 
\end{equation*}
\end{lemma}


Next, we recall from Remark~\ref{rem:SL2R-higgs} (cf.\
\cite{hitchin:1987}) that an $\SL(2,\R)$-Higgs
bundles is defined by a triple $(L,\tilde{\beta},\tilde{\gamma})$
where $L$ is a holomorphic line bundle, $\tilde{\beta}\ne 0\in
H^0(L^2K)$, and $\tilde{\gamma}\in H^0(L^{-2}K)$.  Let $E$ be the
principal $\GL(1,\C)$-bundle which defines $L$.  Using the
identification of $\GL(1,\C)$ with $\SO(2,\C)$ given by
(\ref{eqn:GL1C}), $E$ defines a rank two bundle $L\oplus L^{-1}$.  The
Higgs fields $(\tilde{\beta},\tilde{\gamma})$ then define a bundle map
\begin{equation}\begin{pmatrix}
0 & \tilde{\beta}\\
\tilde{\gamma} & 0
\end{pmatrix}: L\oplus L^{-1}\longrightarrow (L\oplus L^{-1})\otimes K\ .
\end{equation}

\begin{theorem} \label{theorem: SL2R} 
Let
$$\rho_{13}: \SL(2,\R)\longrightarrow\Sp(4,\R)$$
be the irreducible representation as in (\ref{irredJ13}), and let 
$$\varphi: \SL(2,\C)\longrightarrow \Sp(4,\C)$$
be the resulting representation as in (\ref{eqn:varphi}).  Use
$\varphi_{|\GL(1,\C)}$ to extend the structure group of $E$ to
$\GL(2,\C)$ and use $\varphi$ to embed $\liemc(\SL(2,\R))$ in
$\liemc(\Sp(4,\R))$ (cf.\ (\ref{eq:restrict-phi})) .  Let
\begin{equation}
\rho_{ir}^P:\mathcal{M}(\SL(2,\R))\longrightarrow\mathcal{M}(\Sp(4,\R))
\end{equation}

\noi be the induced map from the moduli space of $\SL(2,\R)$-Higgs
bundles to the moduli space of $\Sp(4,\R)$-Higgs bundles . Let
$(L,\tilde{\beta},\tilde{\gamma}) $  be a polystable
$\SL(2,\R)$-Higgs bundle.  Then:

\smallskip

\noi (a) If $0\le \deg(L)\le g-1$ then 
\begin{equation}\label{eqb:rhoirr}
\rho_{ir}^P([L,\tilde{\beta},\tilde{\gamma}])= ([L^3\oplus L^{-1},\beta,\gamma])
\end{equation}

\noi  where 
\begin{equation}\label{eqn:betagamma}
\beta = \begin{pmatrix}
0 &3 \tilde{\beta}\\
3\tilde{\beta}&\tilde{\gamma}\end{pmatrix}  \ ,\  \gamma= \begin{pmatrix}  0 & \tilde{\gamma}  \\
\ \tilde{\gamma} &4\tilde{\beta} 
\end{pmatrix} 
\end{equation}


\noi (b)  If $\deg(L)=g-1$ then $L^2=K$ and $\beta$ and $\gamma$ can be put in the form
\begin{equation}
\gamma=\tilde{\gamma}\begin{pmatrix}0&1\\1&0\end{pmatrix}\ ,\ \beta=\tilde{\gamma}\begin{pmatrix}\beta_1&\beta_3\\
\beta_3& 1\end{pmatrix}\ \mathrm{with}\ \begin{cases} \beta_3 =5 (\frac{\tilde{\beta}}{\tilde{\gamma}})\\\beta_1 = (\frac{16}{25})\beta_3^2\end{cases}
\end{equation}
\end{theorem}

\begin{remark}
The fact that the Higgs bundles obtained in (a) of
Theorem~\ref{theorem: SL2R} are not of the standard form given in Proposition \ref{prop:Vshape}
is due to the fact that unless $\deg(L)=g-1$ the Higgs bundles are not maximal, i.e.\ do not lie in $\mathcal{M}^{max}$.
\end{remark}

\begin{proof}  We use local trivializations and transition functions to describe all bundle data.  Fix an open cover $\{U_i\}$ for $X$ and local trivializations for $L$ and $K$, with transition functions 
$$l_{ij}, k_{ij}:U_i\cap U_j\longrightarrow\GL(1,\C)$$

\noi  on non-empty intersections  $U_i\cap U_j$.  Let the local descriptions of $\tilde{\beta}$ and $\tilde{\gamma}$ over $U_i$ be $\tilde{\beta}_i$ and $\tilde{\gamma}_i$ respectively. Then on non-empty intersections  $U_i\cap U_j$

\begin{equation}\label{eqn: beta rule}
l_{ij}^2k_{ij}\tilde{\beta}_j=\tilde{\beta}_j
\end{equation}

\noindent Similarly 
\begin{equation}\label{eqn: gamma rule general}
l_{ij}^{-2}k_{ij}\tilde{\gamma}_j=\tilde{\gamma}_j\ .
\end{equation}

\noi Observe that if $L^2=K$, so that  $l_{ij}^{2}=k_{ij}$, this implies
\begin{equation}\label{eqn: gamma rule}
\tilde{\gamma}_j=\tilde{\gamma}_j\ .
\end{equation}

The embedding of the $\SL(2,\R)$-Higgs bundle
$(L,\tilde{\beta},\tilde{\gamma})$ in the space of $\Sp(4,\R)$-Higgs
bundles\footnotemark
\footnotetext{To be precise, this yields an $\SL(4,\C)$-Higgs
bundle of the form $(V\oplus V^*, \Phi)$ with
$\Phi=\begin{pmatrix}0& \beta\\ \gamma & 0 \end{pmatrix}$.  The
$\Sp(4,\R)$-Higgs bundle is defined by the data $(V,\beta,\gamma)$.}
is obtained
by applying $\varphi$ to $T^{-1}\begin{pmatrix}l_{ij} & 0 \\ 0& l_{ij}^{-1}\end{pmatrix}T$ and 
$T^{-1} \begin{pmatrix} 0 &\tilde{\beta}_i\\ \tilde{\gamma}_i & 0\end{pmatrix}  T $,where $T=\begin{pmatrix}1& i\\ 1& -i\end{pmatrix}$.  We find
\begin{align*}
\begin{pmatrix}l_{ij} & 0 \\ 0& l_{ij}^{-1}\end{pmatrix}&\longmapsto
\begin{pmatrix}
l_{ij}^{3} & 0 & 0 &0 \\
0 &l_{ij}^{-1} & 0 &0\\
0 &0 &l_{ij}^{-3} & 0\\
0 &0 & 0 & l_{ij}^{1}\\
\end{pmatrix} = g_{ij}\ ,\\
\begin{pmatrix} 0 &\tilde{\beta}_i\\ \tilde{\gamma}_i & 0\end{pmatrix}&\mapsto
\begin{pmatrix}
0 & 0 &  0 & 3\tilde{\beta}_i\\
0 & 0 & 3\tilde{\beta}_i &  \tilde{\gamma}_i \\
0 &\tilde{\gamma}_i   & 0 &0 \\
\ \tilde{\gamma}_i  & 4\tilde{\beta}_i &  0 & 0 
\end{pmatrix} =\Phi_i\ .
\end{align*}

\noi It follows from this that $\{g_{ij}\}$ define a bundle $V\oplus
V^*$ with $V=L^3\oplus L^{-1}$ and that with respect to this
decomposition $\{\Phi_i\}$ define a Higgs field $\Phi$ with $\beta$
and $\gamma$ as in (\ref{eqn:betagamma}).  It remains to show that the
resulting $\Sp(4,\R)$-Higgs bundle, i.e.\ $(L^3\oplus
L^{-1},\beta,\gamma)$, is polystable and thus defines a point in
$\mathcal{M}(\Sp(4,\R))$.

Notice that if $\deg(L)> 0$ and $(L,\tilde{\beta},\tilde{\gamma})$ is
a polystable $\SL(2,\R)$-Higgs bundle, then $\tilde{\gamma}\ne 0$
(cf.\ Remark~\ref{rem:SL2R-stability}). Thus both $\beta$ and $\gamma$
are non-zero.  It follows that $(L^3\oplus L^{-1},\beta,\gamma)$ is
stable if and only if the strict versions of the conditions (3a-c) of
Proposition \ref{prop: n=2stability} are satisfied by line subbundles
$L'\subset L^3\oplus L^{-1}$. But for any such line subbundle, either
$L'=L^3$ or $\deg(L')\le \deg(L^{-1})< 0$. Conditions (3a-c) are thus
clearly satisfied if $L'\ne L^3$. If $L'=L^3$ and $\beta, \gamma$ are
as in (\ref{eqn:betagamma}) then $\beta$ fails to satisfy the
hypotheses in (a) and (c). Moreover, $\gamma$ satisfies the hypothesis
in (b) only if $\tilde{\gamma}=0$, which is not possible if
$(L,\tilde{\beta},\tilde{\gamma})$ is polystable. Thus $L^3$ is not a
destabilizing subbundle and we conclude that $(L^3\oplus
L^{-1},\beta,\gamma)$ is stable.

Finally, if $\deg(L)=0$ then (see Remark~\ref{rem:SL2R-stability})
either $\tilde\beta = \tilde\gamma = 0$ or both $\tilde\beta$ and
$\tilde\gamma$ are non-zero. In the former case, clearly $(L^3\oplus
L^{-1},\beta,\gamma)$ is polystable. In the latter case, clearly the
conditions on $\beta$ and $\gamma$ in (3b-c) of Proposition \ref{prop:
n=2stability} are never satisfied by line subbundles $L'\subset
L^3\oplus L^{-1}$. The only $L'\subset L^3\oplus L^{-1}$ for which the
condition on $\gamma$ in (3a) of Proposition \ref{prop: n=2stability}
is satisfied is $L' = L^3$. But then the condition on $\beta$ in (3a)
of Proposition \ref{prop: n=2stability} is not satisfied and we
conclude that $(L^3\oplus L^{-1},\beta,\gamma)$ is stable. This
completes the proof of part (a).

Suppose now that $\deg(L)=g-1$.  It follows from the definition of
polystability for $\SL(2,\R)$-Higgs bundles that $L^2=K$ and
$\tilde{\gamma}\ne 0$.  By (\ref{eqn: gamma rule}) we can then assume
that the $\tilde{\gamma}_i $ are nowhere zero.  We exploit this to
define an automorphism of $V$ which puts $\gamma$ in a more standard
form.  In the local trivialization over $U_i$, define
\begin{equation}
S_i=\begin{pmatrix} 1 &2\frac{\tilde{\beta}_i}{\tilde{\gamma}_i} & 0 & 0\\
0 & 1 & 0 & 0\\
0 & 0 & 1 & 0\\
0 & 0 & -2\frac{\tilde{\beta}_i}{\tilde{\gamma}_i}  & 1\\
\end{pmatrix}
\end{equation}

\noi Observe that, because of (\ref{eqn: beta rule}) and (\ref{eqn: gamma rule}) we get $g_{ij}S_j g_{ij}^{-1} = S_i$,  which verifies that the $\{S_i\}$ define a bundle automorphism. But
\begin{equation}
S_i\Phi_i S_i^{-1} = \begin{pmatrix}
0 & 0 &16 \frac{\tilde{\beta}^2_i}{\tilde{\gamma}_i} & 5\tilde{\beta}_i\\
0 & 0 & 5\tilde{\beta}_i &  \tilde{\gamma}_i \\
0 &\tilde{\gamma}_i   & 0 &0 \\
\ \tilde{\gamma}_i  & 0 &  0 & 0 
\end{pmatrix}
\end{equation}

\noi Thus the $\Sp(4,\R)$-Higgs bundle defined by $(V,\beta,\gamma)$ is
isomorphic to the $\Sp(4,\R)$-Higgs bundle defined by
$(V,\beta',\gamma')$ where $\beta'$ and $\gamma'$ are as in the
statement of the theorem.
\end{proof}

\begin{corollary}\label{cor:which-hitchin-component}
Let $(V,\beta,\gamma)$ be the image of
$(L,\tilde{\beta},\tilde{\gamma})$ under $\varphi$.
\begin{enumerate}
\item The degree of $V$ is $\deg(V) = 2\deg(L)$.  
\item If $L^2=K$ then $(V,\beta,\gamma)$ lies in the component
$\mathcal{M}^T_L$ of $\mathcal{M}^{max}$.
\end{enumerate}
\end{corollary}

\begin{proof} Part (1) follows immediately from the fact that $V=L^3\oplus
L^{-1}$. For (2), defining $N=L^3$ yields $V=N\oplus N^{-1}K$ with $\deg(N)=3g-3$. This, together with the characterization of $\mathcal{M}^T_{K^{1/2}}$ in Proposition \ref{prop:Vshape}, yields the result.
\end{proof}

\begin{corollary} Let $(V,\beta,\gamma)$ represent a $\Sp(4,\R)$ Higgs
bundles in $\mathcal{M}^T_{K^{1/2}}$ and suppose that it admits a reduction
of structure group to $\SL(2,\R)$.  Then $(V,\beta,\gamma)$ is
isomorphic to a $\Sp(4,\R)$-Higgs bundle with $V=K^{3/2}\oplus K^{-1/2}$
and $\beta$ and $\gamma$ as in Theorem \ref{theorem: SL2R}.
\end{corollary}

\subsection{The normalizer of $\SL(2,\R)$}\label{subs: normirr}
Next we calculate the normalizer of $\SL(2,\R)$ embedded in
$\Sp(4,\R)$ via the irreducible representation.\footnotemark\footnotetext{We are grateful to
Bill Goldman for explaining this to us.}  We shall need the following standard fact.

\begin{proposition}
\label{prop:out-SL2R}
The outer automorphism group of $\SL(2,\R)$ is $\Z/2$, generated by
conjugation by the matrix
$\left(
\begin{smallmatrix}
0 & 1 \\
1 & 0
\end{smallmatrix}\right)$.
\end{proposition}

Consider the extension of the irreducible representation $\rho_1$ to a representation in
$\SL(4,\R)$. Note that the domain of $\rho_1$ can be extended to
$\SL_{\pm}(2,\R)=\{A \suchthat \det(A)
= \pm 1\}$: in fact, substituting $\left(
\begin{smallmatrix}
a & b \\
c & d
\end{smallmatrix}\right)$
by
$\left(
\begin{smallmatrix}
0 & 1 \\
1 & 0
\end{smallmatrix}\right)$
in (\ref{eq:irrepmatrix}) we obtain
\begin{equation}\label{eq:irrep-det-minus}
\rho_1(
\begin{pmatrix}
0 & 1 \\
1 & 0
\end{pmatrix})
=
\begin{pmatrix}
0&0&1&0\\ 0&0&0&1\\ 1&0&0&0\\ 0&1&0&0
\end{pmatrix},
\end{equation}
which has determinant $1$.

Next we make a general observation. Let $\tilde{G} \subset G$ be a Lie
subgroup. We have the following diagram of exact sequences of groups:
\begin{equation}
\label{eq:normalizer-centralizer}
\begin{CD}
&&  1 && 1 && 1 \\
&& @VVV @VVV @VVV \\
1 @>>> Z(\tilde{G}) @>>> \tilde{G} @>>> \Inn(\tilde{G}) @>>> 1 \\
&& @VVV @VVV @VVV \\
1 @>>> C_G(\tilde{G}) @>>> N_G(\tilde{G}) @>>> \Aut(\tilde{G}) \\
&& @VVV @VVV @VVV \\
1 @>>> C_G(\tilde{G})/Z(\tilde{G})
   @>>> N_G(\tilde{G})/\tilde{G} @>>> \Out(\tilde{G}) \\
&& && && @VVV \\
&& && && 1 \\
\end{CD}
\end{equation}
\medskip

\begin{proposition}\label{prop:centralizer-normalizer}
Let $\tilde{G} = \rho_1(\SL(2,\R)) \subset G = \SL(4,\R)$. Then we
have a short exact sequence of groups:
\begin{displaymath}
1 \to C_G(\tilde{G})/Z(\tilde{G})
   \to N_G(\tilde{G})/\tilde{G} \to \Z/2 \to 1,
\end{displaymath}
where the quotient  $\Z/2$ is generated by the image of
$\rho_1(\left(
\begin{smallmatrix}
0 & 1 \\
1 & 0
\end{smallmatrix}\right)) \in N_G(\tilde{G})$.
\end{proposition}

\begin{proof}
As observed above, $\rho_1(\left(
\begin{smallmatrix}
0 & 1 \\
1 & 0
\end{smallmatrix}\right))$
is an element of $G$. Now Proposition~\ref{prop:out-SL2R} implies
that this element belongs to $N_G(\tilde{G})$ and that the map on
the right in the bottom row of (\ref{eq:normalizer-centralizer}) is
surjective. 
\end{proof}

\begin{proposition}\label{prop:centralizer}
Let $\tilde{G} = \rho_1(\SL(2,\R)) \subset G = \SL(4,\R)$. The
centralizer of $\tilde{G}$ in $G$ equals the centre $\{\pm I\}$ of
$\tilde{G}$.
\end{proposition}

\begin{proof}
Any element in the centralizer of $\tilde{G}$ is also in the
centralizer of its complexification. Since this complexification is
just the $4$-dimensional irreducible representation of $\SL(2,\C)$,
Schur's Lemma implies that any element centralizing $\tilde{G}$ is a
complex multiple of the identity. But the only multiples of the
identity in $\SL(4,\R)$ are $\pm I$.
\end{proof}

\begin{corollary}\label{cor:normalizer-SL4R}
The normalizer of $\tilde{G} = \rho_1(\SL(2,\R))$ in $\SL(4,\R)$
fits in the short exact sequence of groups
\begin{displaymath}
1 \to \tilde{G} \to N_{\SL(4,\R)}(\tilde{G}) \to \Z/2 \to 1,
\end{displaymath}
where the quotient $\Z/2$ is generated by the image
$\rho_1(\left(
\begin{smallmatrix}
0 & 1 \\
1 & 0
\end{smallmatrix}\right)) \in N_{\SL(4,\R)}(\tilde{G})$.
\end{corollary}

\begin{proof}
Immediate from Propositions~\ref{prop:centralizer-normalizer} and
\ref{prop:centralizer}. 
\end{proof}

\begin{proposition}\label{prop: Gi=SL2R}
Let $\tilde{G}=\rho_1(\SL(2,\R)) \subset \Sp(4,\R)$. Then the normalizer
of $\tilde{G}$ in $\Sp(4,\R)$, i.e.\ $G_i$,  coincides with
$\tilde{G}$:
\begin{displaymath}
G_i= N_{\Sp(4,\R)}(\tilde{G}) = \tilde{G}.
\end{displaymath}
\end{proposition}

\begin{proof}
Consider $N_{\Sp(4,\R}(\tilde{G}) \subset \Sp(4,\R) \subset
\SL(4,\R)$ as a subgroup of $\SL(4,\R)$. Clearly,
\begin{displaymath}
\tilde{G} \subset N_{\Sp(4,\R}(\tilde{G}) \subset
N_{\SL(4,\R}(\tilde{G}). 
\end{displaymath}
We conclude from Corollary~\ref{cor:normalizer-SL4R} that either
$N_{\Sp(4,\R}(\tilde{G})$ coincides with the index $2$
subgroup $\tilde{G} \subset N_{\SL(4,\R}(\tilde{G})$ or it equals
$N_{\SL(4,\R}(\tilde{G})$. In the latter case, we must have $\rho_1(\left(
\begin{smallmatrix}
0 & 1 \\
1 & 0
\end{smallmatrix}\right)) \in N_{\Sp(4,\R}(\tilde{G})$. But from
(\ref{eq:normalizer-centralizer}) one easily checks that $\rho_1(\left(
\begin{smallmatrix}
0 & 1 \\
1 & 0
\end{smallmatrix}\right))$
does not satisfy (\ref{eq:symp-J_0}) and hence does not belong to
$\Sp(4,\R)$. This concludes the proof.
\end{proof}

\subsection{Summary}

Putting together Theorem~\ref{theorem: SL2R} ,
Corollary~\ref{cor:which-hitchin-component} and the fact that $G_i=\SL(2,\R)$,  we finally obtain:

\begin{theorem}\label{thm:GirrTally}
 A maximal polystable $\Sp(4,\R)$-Higgs bundle deforms to a
 polystable $G_{i}$-Higgs bundle if and only if it belongs to one of
 the Hitchin components $\mathcal{M}^{T}_{K^{1/2}}$.
\end{theorem}

\section{The case $n \geq 3$}
\label{sec:higher-n}

In this section we make a digression to the case of $n \geq 3$,
showing that in this case any maximal polystable $\Sp(2n,\R)$-Higgs
bundle can be deformed to a $G$-Higgs bundle for some proper reductive
Zariski closed subgroup $G \subset \Sp(2n,\R)$.

\subsection{The moduli space of $\Sp(2n,\R)$-Higgs bundles }

An $\Sp(2n,\R)$-Higgs bundle on $X$ (cf.\ Remark~\ref{rem:SL2R-higgs})
is a triple $(V,\beta,\gamma)$, where $V$ is a rank $n$ holomorphic
vector bundle on $X$, $\beta\in H^0(X, S^2V \otimes K)$ and $\gamma
\in H^0(X,S^2V^*\otimes K)$. The \emph{moduli space of polystable
$\Sp(2n,\R)$-Higgs bundles} is denoted by $\mathcal{M}(\Sp(2n,\R))$
and is homeomorphic to the moduli space $\mathcal{R}(\Sp(2n,\R))$ of
reductive representations of $\pi_1(X)$ in $\Sp(2n,\R)$.

The Milnor--Wood inequality for a $\Sp(2n,\R)$-Higgs bundle says that
$\abs{\deg(V)} \leq n(g-1)$. The moduli space of \emph{maximal
$\Sp(2n,\R)$-Higgs bundles} is
\begin{displaymath}
\mathcal{M}^{max}(\Sp(2n,\R)) = \{[V,\beta,\gamma] \in
\mathcal{M}(\Sp(2n,\R)) \suchthat \deg(V)=n \}.
\end{displaymath}
The space $\mathcal{M}^{max}(\Sp(2n,\R))$ is homeomorphic to the
moduli space of maximal representations of $\pi_1(X)$ in $\Sp(2n,\R)$.

For any maximal $\Sp(2n,\R)$-Higgs bundle $(V,\beta,\gamma)$, the map
$\gamma \colon V \to V^* \otimes K$ is an isomorphism and
$(V,\beta,\gamma)$ has a Cayley partner $(W,q_W,\theta)$ defined as in
(\ref{defn:CayleyW})--(\ref{defn:theta}). This leads to the existence
of invariants $w_1(V,\beta,\gamma) \in H^1(X,\Z/2)$ and
$w_2(V,\beta,\gamma) \in H^2(X,\Z/2)$ defined by the Stiefel--Whitney
classes of $(W,q_W)$ (cf.\ \cite{garcia-gothen-mundet:2008}).

The count of connected components of $\mathcal{M}^{max}(\Sp(2n,\R)$
was carried out in \cite{garcia-gothen-mundet:2008}, where the
following theorem is proved.

\begin{theorem}
The moduli space of maximal $\Sp(2n,\R)$-Higgs bundles has $3\cdot
2^{2g}$ connected components:
\begin{enumerate}
\item For each $(w_1,w_2) \in H^1(X,\Z/2) \x H^2(X,\Z/2)$ there is a
 component $\mathcal{M}_{w_1,w_2}$. Any $\Sp(2n,\R)$-Higgs bundle
 $(V,\beta,\gamma)$ in such a component has invariants $(w_1,w_2)$
 and can be deformed to one with $\beta=0$.
\item For each choice of a square root $K^{1/2}$ of the canonical
 bundle of $X$, there is a Hitchin component
 $\mathcal{M}^T_{K^{1/2}}$. Any $\Sp(2n,\R)$-Higgs bundle
 $(V,\beta,\gamma)$ in such a component has
 $\beta\neq 0$ and can be deformed to a $\rho_i(\SL(2,\R))$-Higgs
 bundle, where $\rho_i\colon \SL(2,\R)\to \Sp(2n,\R)$ is the
 irreducible representation. 
\end{enumerate}
\end{theorem}

\begin{remark}
An $\Sp(2n,\R)$-Higgs bundle in $\mathcal{M}_{w_1,w_2}$ has
invariants $(w_1,w_2)$, and an $\Sp(2n,\R)$-Higgs bundle in
$\mathcal{M}^T_{K^{1/2}}$ has $w_2=0$. See
\cite[Proposition~8.2]{garcia-gothen-mundet:2008} for the value of
$w_1$.
\end{remark}

Maximal $\Sp(2n,\R)$-Higgs bundles can be constructed as follows (cf.\
Proposition~\ref{prop: Gpstructure}). Let $(V_i,\beta_i,\gamma_i)$ be
maximal polystable $\Sp(2n_i,\R)$-Higgs bundles for $i=1,2$ and let
$n=n_1+n_2$. Then the polystable $\Sp(2n,\R)$-Higgs bundle $(V,\beta,\gamma)$
defined by
\begin{displaymath}
V = V_1 \oplus V_2,\quad
\beta = \beta_1 + \beta_2,\quad\text{and}\quad
\gamma = \gamma_1 + \gamma_2
\end{displaymath}
is maximal. Of course, such an $\Sp(2n,\R)$-Higgs bundle admits a
reduction of structure group to $\Sp(2n_1,\R)\x\Sp(2n_2,\R)$.

\begin{proposition}
\label{prop:higher-n}
Let $n \geq 3$.

(1) Let $(w_1,w_2)\in H^1(X,\Z/2)\times H^2(X,\Z/2)$ be different
from $(0,1)$. Then there is a maximal $\Sp(2n,\R)$-Higgs bundle which
represents a point in $\mathcal{M}_{w_1,w_2}$ and admits a reduction
of structure group to $\SL(2,\R)\x\dots\x\SL(2,\R)$ ($n$ copies).

(2) There is a maximal $\Sp(2n,\R)$-Higgs bundle which represents a
point in $\mathcal{M}_{0,1}$ and admits a reduction of structure
group to $\Sp(4,\R) \x \SL(2,\R)\x\dots\x\SL(2,\R)$ ($n-2$ copies of
$\SL(2,\R)$).
\end{proposition}

\begin{proof}
(1) This follows by first using the construction in the proof of
Proposition~\ref{w1ne0} to obtain a maximal polystable
$\SL(2,\R)\x\dots\x\SL(2,\R)$-Higgs bundle with the required
$(w_1,w_2)$ and then taking direct sums with $n-2$ copies of a
maximal polystable $\SL(2,\R)$-Higgs bundle with $w_1=0$.  The
$\Sp(2n,\R)$-Higgs bundle $(V,\beta,\gamma)$ thus obtained is
maximal and has invariants $(w_1,w_2)$. Moreover $(V,\beta,\gamma)$
is strictly polystable. Since any $\Sp(2n,\R)$-Higgs bundle in a
Hitchin component is strictly stable
\cite{garcia-prada-gothen-mundet:2009a,hitchin:1992},
$(V,\beta,\gamma)$ does not lie in such a component and it follows that
$(V,\beta,\gamma)$ represents a point in $\mathcal{M}_{w_1,w_2}$ as
required.

(2) Take a maximal polystable $\Sp(4,\R)$-Higgs bundle with
invariants $(w_1,w_2)=(0,1)$ (existence follows from
Proposition~\ref{prop:components}) and take direct sums of this with
$n-2$ copies of a maximal polystable $\SL(2,\R)$-Higgs bundle with
$w_1=0$. As in the proof of (1), we see that this yields a maximal
$\Sp(2n,\R)$-Higgs bundle with the required properties.
\end{proof}

We already knew that maximal $\Sp(2n,\R)$-Higgs bundles in the Hitchin
components can always be deformed to $G'$-Higgs bundles for some proper
Zariski closed subgroup $G'\subset \Sp(2n,\R)$ (namely $G'=\SL(2,\R)$,
embedded via the irreducible representation);
Proposition~\ref{prop:higher-n} tells us that the same is true for
$\Sp(2n,\R)$-Higgs bundles in all other maximal components.
Thus the non-abelian Hodge theory correspondence gives the following.

\begin{corollary}
\label{cor:n-ge-3}
Let $n \geq 3$. Then any maximal representation of $\pi_1(S)$ in
$\Sp(2n,\R)$ can be deformed to one which factors through a proper
reductive Zariski closed subgroup of $\Sp(2n,\R)$.
\end{corollary}

\appendix

\section{The Kronecker product}\label{subs:Sp4R}

If $A$ is an $m\times m$ matrix with entries $a_{ij}$ and $B$ is an
$n\times n$ matrix with entries $b_{ij}$, then the {\bf Kronecker product} $A\otimes B$ is
defined to be the $mn\times mn$ matrix with  block entries $a_{ij}B$.   Thus if $A$ and $B$ are both $2\times 2$ matrices, then
\begin{equation}
A\otimes B=
\begin{pmatrix}
a_{11}B & a_{12}B \\
a_{21}B & a_{22}B
\end{pmatrix}.
\end{equation}

\noindent  Several formulae in the main body of this paper have convenient forms when expressed in terms of this product.  In particular the symplectic forms used to define $\Sp(4,\R)$ are given by
\begin{align}
J_{13}&= J\otimes I,\\
J_{12}&= I\otimes J.
\end{align}

We record some elementary but useful properties of the Kronecker product.   

\begin{lemma}
Let $A, C$ be $m\times m$ matrices and $B,D$ be $n\times n$ matrices. Then
\begin{align}\label{kronecker}
& (A\otimes B)(C\otimes D)=AC\otimes BD\\
& (A\otimes B)^t=A^t\otimes B^t\nonumber\\
& \exp(A\otimes I_n + I_m\otimes B)=\exp(A)\otimes \exp(B)\nonumber
\end{align}

If $A$ and $B$ are both $2\times 2$ matrices and
\begin{equation}\label{eqn:transform}
h=h^t=h^{-1}=\begin{pmatrix}1&0&0&0 \\0&0&1&0
\\0&1&0&0\\0&0&0&1
\end{pmatrix}
\end{equation}

\noi then
\begin{equation}\label{AB=h.BA.h}
A\otimes B = h^t (B\otimes A) h.
\end{equation}
\end{lemma}

\noi Applying (\ref{AB=h.BA.h}) to $J_{13}$ we see that 
\begin{equation}
hJ_{12}=J_{13}h.
\end{equation}

\noi It follows that $g\in \SL(4,\R)$ satisfies $g^tJ_{12}g=J_{12}$ if and only if
$g'=hgh$ satisfies $g'^tJ_{13}g'=J_{13}$. Thus the descriptions of $\Sp(4,\R)$ with respect to $J_{12}$ and with respect to $J_{13}$ are related by conjugation with $h$.

\section{Tables}

\pagestyle{plain}
\begin{landscape}

\begin{table}\label{table: Components}
\begin{tabular}{|c||c||c|c|c||c||c|}
\hline
& & & & & &  \\
Component & Higgs bundle $(V,\beta,\gamma)$& $w_1$ & {\Small $\deg(NK^{-1/2})$} & $w_2$ & $G_*$ & Number\\
\hline
\hline
& &  & & & & \\
&$V=K^{3/2}\oplus K^{-\frac{1}{2}}$ & & & & & \\
& & & & & & \\
$\mathcal{M}^T_{K^{1/2}}$&
$\gamma=\begin{pmatrix}0&1\\1&0\end{pmatrix}\ , \
\beta=\begin{pmatrix}\beta_1&\beta_3\\\beta_3&1\end{pmatrix}$&
$0$ & $2g-2$ & 0& $G_i$ & $2^{2g}$ \\
& & & & & & \\
& $\beta_1\in H^0(K^4)\ , \ \beta_3\in H^0(K^2)$ &  & & & &\\
\hline
& & & & & & \\
& $V=N\oplus N^{-1}K\ ,\  g-1<\mathrm{deg}(N)<3g-3$& & $2g-3$& & &\\
& & &$\vdots$ & & & \\
$\mathcal{M}^0_{c}$& $\gamma=\begin{pmatrix}0&1\\1&0\end{pmatrix}\
,\ \beta=\begin{pmatrix}\beta_1&\beta_3\\ \beta_3&\beta_2\end{pmatrix}\ ,\ \beta_2\ne 0$&
$0$ & $c$& c mod 2 & - & $(2g-3)$\\
({\Small$c=\deg(NK^{-1/2})$})& & &$\vdots$ & & & \\
& $\beta_1\in H^0(N^2K)\ ,\ \beta_3\in H^0(K^2)\ ,\ \beta_2\in H^0(N^{-2}K^3)$ &  & $1$& & &\\
\hline
& & & & & & \\
& $V=N\oplus N^{-1}K\ ,\  \mathrm{deg}(N)=g-1$&   & & & &\\
& & & & & & \\
$\mathcal{M}^0_{0}$& $\gamma=\begin{pmatrix}0&1\\1&0\end{pmatrix}  \ ,\
\beta=\begin{pmatrix}\beta_1&\beta_3\\\beta_3&\beta_2\end{pmatrix}$& 
$0$ & $0$ & $0$ & $G_{\Delta}, G_p$ & 1 \\
& & & & & & \\
& $\beta_1\in H^0(N^2K)\ ,\ \beta_3\in H^0(K^2)\ , \ \beta_2\in H^0(N^{-2}K^3) $ &  & & & & \\
\hline
& & & & & & \\
$\mathcal{M}_{w_1,w_2}$ & $V=W\otimes L_0\ ,\ L_0^2=K$ & & & & &\\
{\tiny $w_1\in H^1(X,\Z/2)-\{0\}$}& & $w_{1}$ & - & $0\ {\mathrm or}\ 1$ &  $G_{\Delta},G_p$  & $2.(2^{2g}-1)$ \\
{\tiny $w_2\in H^2(X,\Z/2)=\Z/2$} & $\gamma=q_W\otimes 1_{L_0}\ ,\ \beta\in H^0(\Sym^2(V)\otimes K)$ &  & & & & \\
\hline
TOTAL & & & & & & $3.2^{2g}+2g-4$\\
\hline
\end{tabular}
\vspace{12pt} \caption{ Higgs bundles in the
components of $\mathcal{M}^{max}$.  The columns show the form of the Higgs bundles, their topological invariants (when applicable), the subgroups to which the structure group of the Higgs bundles can reduce, and the number of connected components of each type. }

\end{table}

\begin{table}\label{table: G*-Higgs}
\begin{tabular}{|c||c|c|c|}
\hline
& & & \\
$G_*$ &  $V$ & $\beta$ & $\gamma$ \\
& & & \\\hline\hline
& & & \\
$G_i$ &  $K^{3/2}\oplus K^{-1/2}$ &$\begin{pmatrix}\beta_1&\beta_3\\\beta_3&1\end{pmatrix}\ ,\ \begin{cases}\beta_3\in H^0(K^2)\\ \beta_1=const.(\beta_3)^2\end{cases}$ 
& $\begin{pmatrix}0&1\\1&0\end{pmatrix}$ \\
& & & \\
\hline
& & & \\
$G_{\Delta}$ & $U\otimes L$ & $q^t_U\otimes \tilde{\beta}$& $q_U\otimes \tilde{\gamma}$ \\
& $U$ orthogonal& $\tilde{\beta}\in H^0(L^{2}K)$& $\tilde{\gamma}\in H^0(L^{-2}K)$  \\
& & & \\
\hline
& & & \\
$\SL(2,\R)\times\SL(2,\R)$  & $L_1\oplus
L_2$ &
$\begin{pmatrix}\beta_1&0\\0&\beta_2\end{pmatrix}$ & $\begin{pmatrix}\gamma_1&0\\0&\gamma_2\end{pmatrix}$ \\
& & & \\
& & & \\
\hline
& & & \\
$G_p$ &  $p^*(V)=L_1\oplus L_2$  & & \\
& $p:X'\longrightarrow X$&
$p^*(\beta)=\begin{pmatrix}\beta_1&0\\0&\beta_2\end{pmatrix}$&
$p^*(\gamma)=\begin{pmatrix}\gamma_1&0\\0&\gamma_2\end{pmatrix}$ \\
& 2:1 & & \\
\hline
\end{tabular}
\vspace{12pt} \caption{$G_*$-Higgs bundles in
$\mathcal{M}^{max}$, showing the special form of the defining data $(V,\beta,\gamma)$ for a $\Sp(4,\R)$-Higgs bundle which admits a reduction of structure group to the indicated subgroup. }\label{tab:G*-higgs}
\end{table}

\end{landscape}


\providecommand{\bysame}{\leavevmode\hbox
to3em{\hrulefill}\thinspace}

\end{document}